\documentclass[a4paper,11pt,reqno]{amsart}
\usepackage{amsmath}
\usepackage{amstext}
\usepackage{amsbsy}
\usepackage{amsopn}
\usepackage{upref}
\usepackage{amsthm}
\usepackage{amsfonts}
\usepackage{amssymb}
\usepackage{mathrsfs}
\usepackage{times}
\usepackage{cite}
\usepackage{bm}
\usepackage{graphicx}
\usepackage{multicol}
\allowdisplaybreaks
\numberwithin{equation}{section}
 \newtheorem{theorem}{Theorem}[section]
 
 \newtheorem{proposition}[theorem]{Proposition}
 \newtheorem{lemma}[theorem]{Lemma}
\theoremstyle{definition}
 \newtheorem{definition}[theorem]{Definition}
\theoremstyle{remark}
 
\begin{document}
\title[$d$-plane transforms]{Microlocal analysis of $d$-plane transform on the Euclidean space}
\author[H.~Chihara]{Hiroyuki Chihara}
\address{College of Education, University of the Ryukyus, Nishihara, Okinawa 903-0213, Japan}
\email{hc@trevally.net}
\thanks{Supported by the JSPS Grant-in-Aid for Scientific Research \#19K03569.}
\subjclass[2020]{Primary 58J40, Secondary 44A12, 53D22}
\keywords{Radon transform, $d$-plane transform, Fourier integral operators, wave front set}
\begin{abstract}
We study the basic properties of $d$-plane transform on the Euclidean space as a Fourier integral operator, and its application to the microlocal analysis of streaking artifacts in its filtered back-projection. The $d$-plane transform is defined by integrals of functions on the $n$-dimensional Euclidean space over all the $d$-dimensional planes, where $0<d<n$. This maps functions on the Euclidean space to those on the affine Grassmannian $G(d,n)$. This is said to be X-ray transform if $d=1$ and Radon transform if $d=n-1$. When $n=2$ the X-ray transform is thought to be measurements of CT scanners. In this paper we obtain concrete expression of the canonical relation of the $d$-plane transform and quantitative properties of the filtered back-projection of the product of the images of the $d$-plane transform. The latter one is related to the metal streaking artifacts of CT images, and some generalization of recent results of Park-Choi-Seo (2017) and Palacios-Uhlmann-Wang (2018) for the X-ray transform on the plane. 
\end{abstract}
\maketitle
\section{Introduction}
\label{section:introduction}
Let $n$ and $d$ be positive integers with $n\geqq2$ and $1\leqq{d}\leqq{n-1}$. We study the microlocal analysis of $d$-plane transform of functions on the $n$-dimensional Euclidean space, and their applications related to metal streaking artifacts of CT images. Following Helgason's celebrated textbook \cite{Helgason} we begin with the definition of the $d$-plane transform and well-known facts on it. Let $G_{d,n}$ be a Grassmannian, that is, the set of all $d$-dimensional vector subspaces of $\mathbb{R}^n$. For $\sigma \in G_{d,n}$ we denote by $\sigma^\perp$ the orthogonal complement of $\sigma$ in $\mathbb{R}^n$. Set 
$$
G(d,n)
:=
\{
x^{\prime\prime}+\sigma 
: 
\sigma \in G_{d,n}, 
x^{\prime\prime} \in \sigma^\perp
\},
$$
which is the set of all $d$-dimensional planes of $\mathbb{R}^n$ and is said to be an affine Grassmannian. We sometimes denote $x^{\prime\prime}+\sigma \in G(d,n)$ by $(\sigma,x^{\prime\prime})$. Here we introduce the $d$-plane transform $\mathcal{R}_d$. 
For a function $f(x)=\mathcal{O}(\langle{x}\rangle^{-d-\varepsilon})$ of 
$x=(x_1,\dotsc,x_n)=x^\prime+x^{\prime\prime}\in\sigma\oplus\sigma^\perp=\mathbb{R}^n$ 
with $\sigma \in G_{d,n}$, we define $\mathcal{R}_df$ by 
$$
\mathcal{R}_df(\sigma,x^{\prime\prime})
:=
\int_\sigma
f(x^\prime+x^{\prime\prime})
dx^\prime,
$$
where 
$\langle{x}\rangle=\sqrt{1+\lvert{x}\rvert^2}$, 
$\lvert{x}\rvert^2=x{\cdot}x$, 
$x{\cdot}y$ is the standard inner product of $x,y\in\mathbb{R}^n$, 
and 
$dx^\prime$ is the $d$-dimensional Lebesgue measure on $\sigma$. 
We also use similar notation $\langle\xi;\eta\rangle=\sqrt{1+\lvert\xi\rvert^2+\lvert\eta\rvert^2}$. 
$\mathcal{R}_1f$ is said to be the X-ray transform of $f$, 
and $\mathcal{R}_{n-1}f$ is said to be the Radon transform of $f$. 
\par
Next we recall inversion formula of $\mathcal{R}_d$. 
Roughly speaking, the formal adjoint of $\mathcal{R}_d$ is given as integrals of functions over the set of all $d$-planes passing through arbitrary fixed point $x\in\mathbb{R}^n$. 
More precisely 
\begin{align*}
  \mathcal{R}_d^\ast\varphi(x)
& :=
  \frac{1}{C(d,n)}
  \int_{\{\Xi \in G(d,n) : x \in \Xi\}}
  \varphi(\Xi)
  d\mu(\Xi)
\\
& =
  \frac{1}{C(d,n)}
  \int_{O(n)}
  \varphi(x+k\cdot\sigma)
  dk,
\end{align*}
where 
$\varphi \in C\bigl(G(d,n)\bigr)$, 
$C\bigl(G(d,n)\bigr)$ is the set of all continuous functions on $G(d,n)$, 
$C(d,n)=(4\pi)^{d/2}\Gamma(n/2)/\Gamma\bigl((n-1)/2\bigr)$, 
$\Gamma(\cdot)$ is the gamma function, 
$O(n)$ is the orthogonal group, 
$d\mu$ and $dk$ are the normalized measures which are invariant under rotations, 
and $\sigma \in G_{d,n}$ is arbitrary. 
The inversion formula of $\mathcal{R}_d$ is given as follows. 
\begin{proposition}[{\cite[Theorem~6.2]{Helgason}}]
\label{theorem:fbp}
For $f(x)=\mathcal{O}(\langle{x}\rangle^{-d-\varepsilon})$ 
$$
f
=
(-\Delta_x)^{d/2}\mathcal{R}_d^\ast\mathcal{R}_df
=
\mathcal{R}_d^\ast(-\Delta_{x^{\prime\prime}})^{d/2}\mathcal{R}_df, 
$$
where 
$-\Delta_x=-\partial_{x_1}^2-\dotsb-\partial_{x_n}^2$ 
and
$-\Delta_{x^{\prime\prime}}$ is the Laplacian on $\sigma^\perp$. 
\end{proposition}
Operators $\mathcal{R}_d^\ast$ and 
$(-\Delta_x)^{d/2}\mathcal{R}_d^\ast=\mathcal{R}_d^\ast(-\Delta_{x^{\prime\prime}})^{d/2}$ 
are said to be the unfiltered back-projection operator 
and the filtered back-projection operator respectively. 
\par
Here we recall the mapping properties of $\mathcal{R}_d$. 
To state them we introduce some function spaces. 
Let $X$ be a smooth manifold. 
We denote by $C^\infty(X)$ the set of all smooth functions on $X$. 
The set of all compactly supported smooth functions on $X$ is denoted by $C^\infty_0(X)$. 
$\mathcal{D}(X)$ is the locally convex space defined by adding the inductive limit topology to $C^\infty_0(X)$. We denote by $\mathcal{D}^\prime(X)$ the space of Schwartz distributions on $X$, which is the topological dual of $\mathcal{D}(X)$.  
We denote by $\mathcal{D}_H\bigl(G(d,n)\bigr)$ 
the set of all $\varphi \in \mathcal{D}\bigl(G(d,n)\bigr)$ 
satisfying the following moment conditions: 
for any $k=0,1,2,\dotsc$ there exists a homogeneous polynomial $P_k$ 
on $\mathbb{R}^n$ of degree $k$ such that for any $\sigma \in G_{d,n}$ 
$$
P_k(\xi^{\prime\prime})
=
\int_{\sigma^\perp}
(\xi^{\prime\prime} \cdot x^{\prime\prime})^k
\varphi(\sigma,x^{\prime\prime})
dx^{\prime\prime},
\quad
\xi^{\prime\prime} \in \sigma^\perp.
$$    
The range of $\mathcal{R}_d$ of $\mathcal{D}(\mathbb{R}^n)$ is characterized as follows. 
\begin{proposition}[{\cite[Theorem~6.3]{Helgason}}]
\label{theorem:range}
$\mathcal{R}_d$ is a linear bijection of $\mathcal{D}(\mathbb{R}^n)$ 
onto $\mathcal{D}_H\bigl(G(d,n)\bigr)$.
\end{proposition}
\par
It is well-known that $\mathcal{R}_d$ is an elliptic Fourier integral operator. 
See \cite{GuilleminSternberg}. 
Moreover the canonical relation of $\mathcal{R}_{n-1}$ 
is well-known and applied to the microlocal analysis of CT images for $n=2$. 
See 
\cite{Chihara}, 
\cite{KrishnanQuinto}, 
\cite{PalaciosUhlmannWang}, 
\cite{ParkChoiSeo}, 
\cite{Quinto1} 
and 
\cite{Quinto2}.  
As far as the author knows the canonical relation of $\mathcal{R}_d$ for $d{\ne}n-1$ was rarely considered so far. It is important to obtain a comprehensive and useful expression of them. Indeed the canonical relation of the X-ray transform $\mathcal{R}_1$ on $\mathbb{R}^3$ is expected to play a crucial role in the microlocal analysis of the CT scanners of the fourth generation or later ones. See, e.g., \cite{Epstein} for the history of actual CT scanners including the fourth generation until 2002. 
\par
Following Epstein's textbook \cite{Epstein} on mathematical theory of medical imaging and a recent paper \cite{ParkChoiSeo} of Park, Choi and Seo, we review the outline of computed tomography (CT). In what follows we assume for a while that $n=2$ and $d=1(=n-1)$. Throughout the present paper the X-ray beam is supposed to have no width. In medical imaging $f(x)$ and $\mathcal{R}_1f$ are thought to be the distribution of attenuation coefficients of the section of a human body and the measurements $P(\sigma,x^{\prime\prime})$ of a CT scanner for this section. If a human body consists only of normal tissue, then it follows from the Beer-Lambert law that $P=\mathcal{R}_1f$ holds, and $f$ is reconstructed by the filtered back-projection. However there are some factors causing artifacts in actual CT images: beam width, partial volume effect, beam hardening, noise in measurements, numerical errors and etc. We focus on metal streaking artifacts caused by the beam hardening. If a human body contains metal such as implants, dental fillings, stents, metal bones and etc, streaking artifacts occur on lines tangent to two or more boundaries of metal region in its CT images. This is basically due to the fluctuation of the energy level $E \in [0,\infty)$ of X-ray beams for normal tissue and metal regions. We explain this more concretely. In what follows we denote the distribution of attenuation coefficients by $f_E(x)$ with $E \in [0,\infty)$. The measurements of CT scanners are described by a spectral function $\rho(E)$, which is a probability density function on $[0,\infty)$:
$$
P
=
-
\log
\left\{
\int_0^\infty
\rho(E)
\exp(-\mathcal{R}_1f_E)
dE
\right\}.
$$
If a human body consists only of normal tissue, we may assume that the X-ray beams are monochromatic with a fixed energy $E_0>0$, that is, $f_E$ is independent of $E$ and $f_E=f_{E_0}$. In this case the measurements become
$$
P
=
-
\log
\left\{
\exp(-\mathcal{R}_1f_{E_0})
\cdot
\int_0^\infty
\rho(E)
dE
\right\}
=
\mathcal{R}_1f_{E_0}.
$$
The beam hardening gives nonlinear effects to the measurements of CT scanners and cause the metal streaking artifacts in CT images. 
\par
Recently the metal streaking artifacts in CT images were studied by using microlocal analysis. To explain this we introduce some setting and notation. Let $D$ be the metal region in $\mathbb{R}^2$. We remark that the singular support of the characteristic function of $D$ is its boundary $\partial{D}$ and the singular directions are the normal directions on $\partial{D}$. Consider the case that the spectral function is a normalized characteristic function of an interval $[E_0-\varepsilon,E_0+\varepsilon]$, that is, 
$$
\rho(E)
=
\frac{1}{2\varepsilon}
\chi_{[E_0-\varepsilon,E_0+\varepsilon]}(E). 
$$
Suppose that $f_E(x)$ is a sum of the normal tissue part $f_{E_0}(x)$ and the metal part of the form 
$$
f_E(x)
=
f_{E_0}(x)
+
\alpha(E-E_0)\chi_D(x)
$$
with some small parameter $\alpha>0$. 
If $\alpha\varepsilon>0$ is small enough, the measurements become
\begin{align*}
  P
& =
  \mathcal{R}_1f_{E_0}
  -
  \log\left\{\frac{\sinh(\alpha\varepsilon\mathcal{R}_1\chi_D)}{\alpha\varepsilon\mathcal{R}_1\chi_D}\right\}
\\
& =
  \mathcal{R}_1f_{E_0}
  +
  \sum_{l=1}^\infty
  A_l(\alpha\varepsilon\mathcal{R}_1\chi_D)^{2l}
\end{align*}
with some sequence of real numbers $A_1,A_2,A_3,\dotsc$. 
In their pioneering work \cite{ParkChoiSeo} Park, Choi and Seo first studied the microlocal analysis of this phenomenon. Roughly speaking they proved that if there are two points on $\partial{D}$ which are connected by a common tangential line, then the streaking artifacts appear on the tangential line in the filtered back-projection of $P$. Note that in this case the singular directions of the two points are the same. They also described the streaking artifacts as the wave front set of the filtered back-projection of $P$. Their results show the qualitative properties of the metal streaking artifacts. Moreover in \cite{PalaciosUhlmannWang} Palacios, Uhlmann and Wang developed the results of \cite{ParkChoiSeo}. By using the notion of paired Lagrangian distributions introduced in \cite{GreenleafUhlmann}, they proved that the metal streaking artifacts are conormal distributions supported on all the common tangential lines. These results are the quantitative properties of the metal streaking artifacts. These results are concerned with the case that $D$ is a disjoint union of finite number of convex domains. More recently, Wang and Zou in \cite{WangZou} studied the case that $D$ is a non-convex planar domain. They proved that streaking artifacts are contained not only in common tangential lines but also in tangential lines with only one inflection point. The latter case is essentially due to the non-convexity of $D$.    
\par
Here we show Figure~1 illustrating a characteristic function $f$ of two disks on the plane, its X-ray transform $\mathcal{R}_1f$, and the filtered back-projections of $\mathcal{R}_1f$ and $(\mathcal{R}_1f)^2$. We see the streaking artifacts on the four lines tangent to both two disks in the filtered back-projection of $(\mathcal{R}_1f)^2$. 
\begin{center}
\includegraphics[height=65mm]{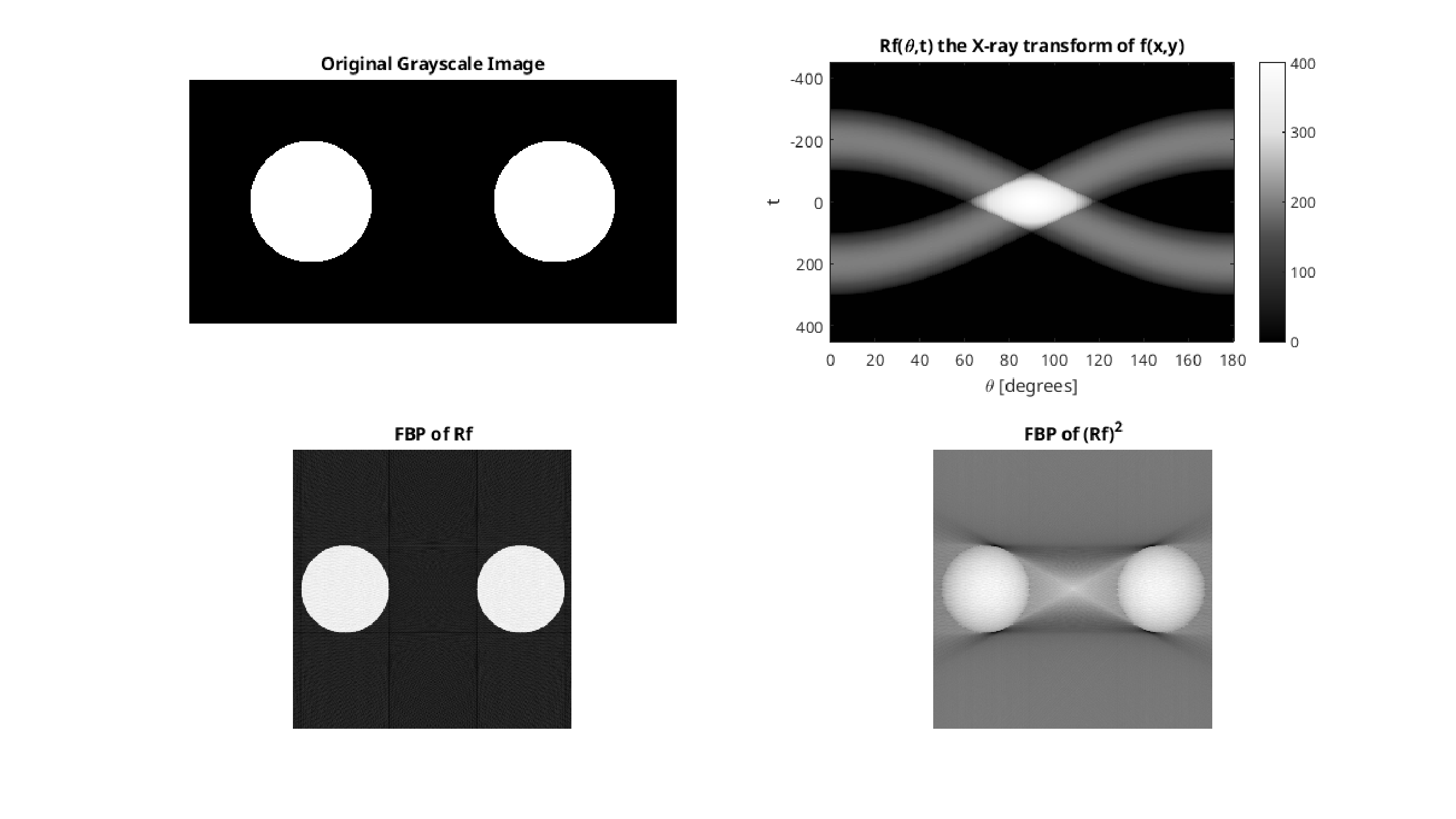} 
\\
Figure~1.
\quad 
\end{center}
\par
In the present paper we shall obtain the canonical relations for all the $d$-plane transforms on $\mathbb{R}^n$ for general $n=2,3,4,\dotsc$ and $0<d<n$. Our expression of them is useful to understand the relationship between some conormal bundles in the cotangent bundles of $\mathbb{R}^n$ and $G(d,n)$. Based on the canonical relation we shall also study the metal streaking artifacts related to them. The organization of the present paper is as follows. In Section~\ref{section:micro} we quickly review microlocal analysis used in the present paper. In Section~~\ref{section:dplane} we obtain a concrete expression of the cotangent bundle of $G(d,n)$, and compute the canonical relation of $\mathcal{R}_d$. In Section~\ref{section:metalregion} we study the geometry of the conormal bundle of $\partial{D}$ and its canonical transformation associated to $\mathcal{R}_d$ under the assumption that $D$ is a disjoint union of finite number of strictly convex bounded domains with smooth boundaries. In Section~\ref{section:lagrangian}, following \cite{GreenleafUhlmann} we introduce paired Lagrangian distributions and study the product $\mathcal{R}_d\chi_D\cdot\mathcal{R}_d\chi_D$. Finally in Section~\ref{section:artifacts} we prove that the metal streaking artifacts are conormal distributions whose singular supports are the union of hypersurfaces which are cone surfaces or cylinder surfaces and contact boundaries of more than two metal regions. In other words, our main theorem says that if two metal regions have a common tangential hyperplane, then the normal direction of the tangent points $P$ and $Q$ is the only one singular direction at $P$ and $Q$, and this microlocal singularity propagates along the line connecting $P$ and $Q$. We believe that our results are natural generalization of recent results of \cite{ParkChoiSeo} and \cite{PalaciosUhlmannWang} for the X-ray transform on the two-dimensional plane. The case that $D$ is a disjoint union of finite numbers of convex polygons in $\mathbb{R}^2$ was also studied in \cite{ParkChoiSeo} and \cite{PalaciosUhlmannWang}. It is possible to study the case that $D$ is a disjoint union of finite numbers of convex polyhedra in $\mathbb{R}^3$. However we omit this since we need to take many kinds of relations such as vertex-vertex, vertex-edge, vertex-polygon, edge-edge, edge-polygon and polygon-polygon into account.  
\par
Our method of analysis of the streaking artifacts is based on the idea of \cite{PalaciosUhlmannWang}. Throughout the present paper we make use of microlocal analysis freely. See H\"ormander's four volumes of celebrated textbooks on microlocal analysis \cite{Hoermander1}, \cite{Hoermander2}, \cite{Hoermander3} and \cite{Hoermander4}, and textbooks on Fourier integral operators \cite{Duistermaat} and \cite{GrigisSjoestrand}.
\section{Quick review of microlocal analysis}
\label{section:micro} 
We shall briefly review some basic notions of microlocal analysis used later. 
We begin with the definition of the Fourier transform in $\mathbb{R}^N$ with $N=1,2,3,\dotsc$. 
For a rapidly decreasing function $f(x)$ on $\mathbb{R}^N$, its Fourier transform is defined by 
$$
\hat{f}(\xi)
=
\mathcal{F}f(\xi)
:=
\int_{\mathbb{R}^N}
e^{-ix\cdot\xi}
f(x)
dx, 
\quad
\xi\in\mathbb{R}^N, 
$$
and the inverse Fourier transform of a rapidly decreasing function $g(\xi)$ on $\mathbb{R}^N$ is defined by 
$$
\mathcal{F}^\ast{g}(x)
:=
(2\pi)^{-N}\mathcal{F}g(-x)
=
\frac{1}{(2\pi)^N}
\int_{\mathbb{R}^N}
e^{ix\cdot\xi}
g(\xi)
d\xi, 
\quad
x\in\mathbb{R}^N, 
$$
Here we introduce wave front sets of distributions. 
\begin{definition}[Wave front set]
\label{theorem:wavefrontset}
Let $X$ be a smooth manifold. $T^\ast{X}$ denotes the cotangent bundle of $X$. 
For $u\in\mathscr{D}^\prime(X)$ and $(x,\xi) \in T^\ast{X}\setminus0$,  
we say that $(x,\xi) \not\in \operatorname{WF}(u)$ 
if there exists $\phi(y) \in C^\infty_0(X)$ with $\phi(x)\ne0$ and a conic neighborhood of $V$ at $\eta=\xi$ such that for any $M>0$
$$
\widehat{\phi{u}}(\eta)
=
\mathcal{O}(\langle\eta\rangle^{-M})
\quad
\text{for}
\quad
\eta\in{V}. 
$$
We denote this by 
$\widehat{\phi{u}}(\xi)=\mathcal{O}(\langle\xi\rangle^{-\infty})$. 
$\operatorname{WF}(u)$ is said to be the wave front set of $u$. 
\end{definition}
Equivalently, if we set 
$$
\Sigma_x(u)
:=
\bigcap_{\substack{\phi(y) \in C^\infty_0(X) \\ \phi(x)\ne0}}
\{
\eta{\in}T_x^\ast{X}\setminus0 
: 
\widehat{\phi{u}}(\eta)
\ne
\mathcal{O}(\langle\eta\rangle^{-\infty})
\},
$$
then we have 
$$
\operatorname{WF}(u)
=
\{
(x,\xi){\in}T^\ast{X}\setminus0 
: 
\xi\in\Sigma_x(u)
\}.
$$
Consider a characteristic function $g(x,y)$ of a half plane in $\mathbb{R}^2$: 
$g(x,y)=1$ if $x>0$ and $g(x,y)=0$ otherwise. 
For $\phi(x,y) \in C^\infty_0(\mathbb{R}^2)$ we deduce that 
\begin{itemize}
\item 
If $\phi(0,y)=0$, then 
$\widehat{\phi{g}}(\xi,\eta)=\hat{\phi}(\xi,\eta)$ 
or 
$\widehat{\phi{g}}(\xi,\eta)=0$. We have 
$\widehat{\phi{g}}(\xi,\eta)=\mathcal{O}(\langle\xi;\eta\rangle^{-\infty})$, 
where 
$\langle\xi;\eta\rangle=\sqrt{1+\lvert\xi\rvert^2+\lvert\eta\rvert^2}$. 
\item 
If $\phi(0,y)\ne0$, then 
$$
\widehat{\phi{g}}(\xi,\eta)
=
\frac{\hat{\phi}(\xi,\eta)}{2}
+
\frac{1}{2i}
\operatorname{VP}
\int_{\mathbb{R}}
\frac{\hat{\phi}(\zeta,\eta)}{\xi-\zeta}
d\zeta
=
\mathcal{O}(\langle\eta\rangle^{-\infty}).
$$
\end{itemize}
Hence we have $\operatorname{WF}(g)=\{(0,y,\xi,0)\ \vert \ y,\xi\in\mathbb{R}, \xi\ne0\}$. $g(x,y)$ is a typical example of conormal distributions. Using this notion we can describe and compute microlocal singularities quantitatively. Conormal distributions and their generalization called Lagrangian distributions play crucial role in the present paper. Following H\"ormander's textbook \cite[Section~18.2]{Hoermander3} we introduce basic facts on conormal distributions. To state the definition of conormal distributions we introduce function spaces on $\mathbb{R}^N$. Split $\mathbb{R}^N$ into a disjoint family of a disk $X_0$ and annuli $X_1,X_2,X_3,\dotsc$ defined by 
$$
X_0:=\{\lvert\xi\rvert<1\}, 
\qquad
X_j:=\{2^{j-1}\leqq\lvert\xi\rvert<2^j\}, 
\quad
(j=1,2,3,\dotsc).
$$
Let $s$ be a real number. We say that a tempered distribution $u$ on $\mathbb{R}^N$ belongs to 
${}^\infty{H}_{(s)}(\mathbb{R}^N)$ if
$$
\lVert{u}\rVert_{{}^\infty H_{(s)}(\mathbb{R}^N)}
:=
\sup_{j=0,1,2,\dotsc}
\left(
\int_{X_j}
\langle\xi\rangle^{2s}
\lvert\hat{u}(\xi)\rvert^2
d\xi
\right)^{1/2}<\infty. 
$$
Using smooth cut-off functions, we can define the local space 
${}^\infty{H}_{(s)}^\text{loc}(X)$ 
of distributions on an $N$-dimensional smooth manifold $X$. 
Here we define conormal distributions. 
\begin{definition}[Conormal distribution]
Let $X$ be an $N$-dimensional smooth manifold, and let $Y$ be a closed submanifold of X. 
$u \in \mathscr{D}^\prime(X)$ is said to be conormal with respect to $Y$ of degree $m$ if 
$$
L_1{\dotsb}L_Mu 
\in 
{^\infty}H^\text{loc}_{(-m-N/4)}(X)
$$
for all $M=0,1,2,\dotsc$ and all vector fields $L_1,\dotsc,L_M$ tangential to $Y$. 
Denote by $I^m(X,N^\ast{Y})$ or $I^m(N^\ast{Y})$, 
the set of all distributions on $X$ which are conormal with respect to $Y$ of degree $m$. 
Here $N^\ast{Y}$ is the conormal bundle of $Y$, which is a conic Lagrangian submanifold of $T^\ast{X}$. 
\end{definition}
Roughly speaking, 
$I^m(X,N^\ast{Y})$ is the set of all distributions $u$ on $X$ such that 
$\operatorname{singsupp}u \subset Y$ and $u$ has some upper bound of the order of singularities whose directions are perpendicular to $Y$. Conormal distributions can be characterized by oscillatory integrals which are similar to the distribution kernels of pseudodifferential operators. Here we introduce symbol classes. 
\begin{definition}
\label{theorem:symbol} 
For $m\in\mathbb{R}$ we denote by $S^m(\mathbb{R}^N\times\mathbb{R}^k)$ the set of all smooth functions $a(x,\xi^\prime)$ on 
$\mathbb{R}^N\times\mathbb{R}^k$ such that for any compact set $K$ in $\mathbb{R}^N$ and for any multi-indices 
$\alpha$ and $\beta$ 
$$
\lvert\partial_x^\beta\partial_{\xi^\prime}^\alpha a(x,\xi^\prime)\rvert
\leqq
C_{K,\alpha,\beta}\langle\xi^\prime\rangle^{m-\lvert\alpha\rvert},
\quad
(x,\xi^\prime) \in K\times\mathbb{R}^k
$$
with some positive constant $C_{K,\alpha,\beta}$. 
\end{definition}
Here we state the characterization of conormal distributions. 
\begin{proposition}[{\cite[Theorem~18.2.8]{Hoermander3}}]
\label{theorem:characterizeconormal}
Let $X$ be an $N$-dimensional smooth manifold, 
and let $Y$ be a closed submanifold of $X$ with $\operatorname{codim}Y=k$. 
Fix arbitrary $y \in Y$ and choose local coordinates 
$x=(x^\prime,x^{\prime\prime})\in\mathbb{R}^k\times\mathbb{R}^{N-k}$ 
such that $x(y)=0$ and $Y$ is represented as $\{x^\prime=0\}$ locally. 
For $u \in I^{m+k/2-N/4}(N^\ast{Y})$ there exists an amplitude 
$a(x^{\prime\prime},\xi^\prime) \in S^m(\mathbb{R}^{N-k}\times\mathbb{R}^k)$ such that 
\begin{equation}
u(x)
=
\int_{\mathbb{R}^k}
e^{ix^\prime\cdot\xi^\prime}
a(x^{\prime\prime},\xi^\prime)
d\xi^\prime
\quad
\text{near}
\quad
x=0. 
\label{equation:oscillation1}
\end{equation}
Conversely every $u$ of this form is in $I^{m+k/2-N/4}(N^\ast{Y})$. 
\end{proposition}
Here we remark that if there exists an amplitude 
$b(x,\xi^\prime) \in S^m(\mathbb{R}^N\times\mathbb{R}^k)$ such that 
\begin{equation}
u(x)
=
\int_{\mathbb{R}^k}
e^{ix^\prime\cdot\xi^\prime}
b(x,\xi^\prime)
d\xi^\prime
\quad
\text{near}
\quad
x=0. 
\label{equation:oscillation2}, 
\end{equation}
then this can be reduced to \eqref{equation:oscillation1}. 
Indeed, if \eqref{equation:oscillation2} holds, 
then \eqref{equation:oscillation1} holds with an amplitude given by  
$$
a(x^{\prime\prime},\xi^\prime)
:=
e^{-iD_{x^\prime} \cdot D_{\xi^\prime}}b(x,\xi^\prime)\vert_{x^\prime=0}
\sim
\sum_{j=0}^\infty
\frac{(-iD_{x^\prime} \cdot D_{\xi^\prime})^jb(x,\xi^\prime)\vert_{x^\prime=0}}{j!},  
$$
where $D_{x^\prime}=-i\partial_{x^\prime}$ and $\sim$ means the asymptotic expansion of symbols. 
See \cite[Lemma~18.2.1]{Hoermander3}. 
\par
We apply Proposition~\ref{theorem:characterizeconormal} to two elementary examples. 
The characteristic function $g(x,y)$ of the half plane in $\mathbb{R}^2$ belongs to 
$$
I^{-1}\bigl(N^\ast\{(0,y) : y\in\mathbb{R}\}\bigr)
=
I^{-1}\bigl(\mathbb{R}^2,\{(0,y,\xi,0)\} : y,\xi\in\mathbb{R}\bigr). 
$$  
Consider a pseudodifferential operator 
$$
\operatorname{Op}(a)u(x)
=
\frac{1}{(2\pi)^N}
\iint_{\mathbb{R}^N\times\mathbb{R}^N}
e^{i(x-y)\cdot\xi}
a(x,\xi)u(y)
dyd\xi, 
\quad
u \in \mathcal{D}(\mathbb{R}^N)
$$
with $a(x,\xi) \in S^m(\mathbb{R}^N\times\mathbb{R}^N)$. 
Set $\Delta=\{(x,x) : x\in\mathbb{R}^N\}$. 
This is the diagonal part of $\mathbb{R}^N\times\mathbb{R}^N$.  
The distribution kernel of $\operatorname{Op}(a)$ is given by an oscillatory integral 
$$
K(x,y)
=
\frac{1}{(2\pi)^N}
\int_{\mathbb{R}^N}
e^{i(x-y)\cdot\xi}
a(x,\xi)
d\xi
$$
and belongs to 
$$
I^m(N^\ast\Delta)
=
I^m\bigl(\mathbb{R}^N\times\mathbb{R}^N,\{(x,x,\xi,-\xi) : x,\xi\in\mathbb{R}^N\}\bigr).
$$
\par
Conormal distributions are typical examples of Lagrangian distributions, and pseudodifferential operators are typical examples of Fourier integral operators. The distribution kernels of Fourier integral operators are Lagrangian distributions. We need more general notions like them later since the $d$-plane transform is an elliptic Fourier integral operator. In what follows we shall quickly introduce Lagrangian distributions, and Fourier integral operators and their canonical relations. We begin with phase functions which are general form of $(x-y)\cdot\xi$ 
in $e^{i(x-y)\cdot\xi}$ for pseudodifferential operators. 
\begin{definition}[non-degenerate phase function]
\label{theorem:phase}
Let $X$ be an $N$-dimensional smooth manifold. 
We say that a real-valued function 
$\varphi(x,\theta) \in C^\infty\bigl(X{\times}(\mathbb{R}^k\setminus\{0\})\bigr)$ 
is a non-degenerate phase function if 
\begin{itemize}
\item 
$\varphi(x,t\theta)=t\varphi(x,\theta)$ 
for 
$(x,\theta)\in X{\times}(\mathbb{R}^k\setminus\{0\})$ 
and 
$t>0$. 
\item 
$d\varphi(x,\theta)\ne0$. 
\item 
If $\varphi^\prime_\theta(x,\theta)=0$, 
then 
$\operatorname{rank}[d\varphi^\prime_\theta(x,\theta)] \equiv k$. 
\end{itemize}
\end{definition}
In this case 
$C_\varphi=\{(x,\theta) : \varphi^\prime_\theta(x,\theta)=0\}$ 
is an $N$-dimensional submanifold of $X\times(\mathbb{R}^k\setminus\{0\})$. 
Set $\Lambda_\varphi:=\bigl\{\bigl(x,\varphi^\prime_x(x,\theta)\bigr) : \varphi^\prime_\theta(x,\theta)=0\bigr\}$. Then 
$$
C_\varphi\ni(x,\theta) \mapsto \bigl(x,\varphi^\prime_x(x,\theta)\bigr)\in\Lambda_\varphi
$$
is a diffeomorphism and $\Lambda_\varphi$ 
becomes a conic Lagrangian submanifold of $T^\ast{X}\setminus0$, that is, 
a conic submanifold with $d\xi{\wedge}dx\equiv0$ on $\Lambda_\varphi$. 
Here we introduce Lagrangian distributions.  
\begin{definition}[Lagrangian distribution]
Let $\Lambda$ be a conic Lagrangian submanifold of $T^\ast{X}\setminus0$. 
We denote by $I^m(X,\Lambda)=I^m(\Lambda)$ the set of all $u\in\mathscr{D}^\prime(X)$ satisfying 
\begin{itemize}
\item 
$\operatorname{WF}(u)\subset\Lambda$. 
\item 
For any $(x_0,\xi_0)\in\Lambda$ 
there exist a non-degenerate phase function $\varphi(x,\theta)$ 
and an amplitude $a(x,\theta) \in S^{m+N/4-k/2}(X\times\mathbb{R}^k)$ with some $k=1,2,3,\dotsc$ 
such that $\Lambda=\Lambda_\varphi$ near $(x_0,\xi_0)$ and 
$$
u(x)
=
\int e^{i\varphi(x,\theta)}a(x,\theta)d\theta
\quad
\text{microlocally near}
\quad
(x_0,\xi_0).
$$
\end{itemize}
\end{definition}
We now introduce Fourier integral operators and their properties needed later. Let $X$ and $Y$ be manifolds, and let $\Lambda$ be a conic Lagrangian submanifold of 
$T^\ast(X{\times}Y)\setminus0$. For $K(x,y) \in I^m(X{\times}Y,\Lambda)$, 
$$
Au(x):=\int_YK(x,y)u(y)dy, 
\quad
u\in\mathscr{D}(Y)
$$
defines a linear operator of $\mathscr{D}(Y)$ to $\mathscr{D}^\prime(X)$.  
$A$ is said to be a Fourier integral operator. 
Indeed $A$ is locally given by 
$$
Au(x)
:=
\iint
e^{i\varphi(x,y,\theta)}
a(x,y,\theta)u(y)
dyd\theta
$$
with a non-degenerate phase function $\varphi(x,y,\theta)$ and an amplitude $a(x,y,\theta)$. 
If the principal part of $a$ does not vanish, then $A$ is called an elliptic Fourier integral operator. 
\par
$\Lambda^\prime:=\{(x,y;\xi,\eta) : (x,y;\xi,-\eta)\in\Lambda\}$ 
is said to be the canonical relation of $A$. We have 
$\operatorname{WF}(Au) \subset \Lambda^\prime\circ\operatorname{WF}(u)$ 
and if $A$ is elliptic then 
$\operatorname{WF}(Au) = \Lambda^\prime\circ\operatorname{WF}(u)$ 
holds. Here $\circ$ denotes the composition defined as follows. 
For $W \subset T^\ast{X} \times T^\ast{Y}$ and $Z \subset T^\ast{Y}$ we set 
$$
W{\circ}Z
:=
\{(x,\xi) \in T^\ast{X} : (x,y;\xi,\eta) \in W \ \text{with some}\ (y,\eta) \in Z\}. 
$$
The canonical relation describes the correspondence between the microlocal singularities of $u$ and those of $Au$. 
In other words this correspondence is a mapping 
\begin{equation}
T^\ast{Y}
\ni
\bigl(y,-\varphi^\prime_y(x,y,\theta)\bigr)
\mapsto 
\bigl(x,\varphi^\prime_x(x,y,\theta)\bigr)
\in
T^\ast{X}. 
\label{equation:canonicaltransformation}
\end{equation}
This is called a canonical transformation. 
\par
Formally the adjoint of $A$ is given by 
$$
A^\ast{v(y)}
=
\iint
e^{-i\varphi(x,y,\theta)}
\overline{a(x,y,\theta)}
v(x)
dxd\theta. 
$$
Hence if $\Lambda^\prime$ is the canonical relation of $A$, then the phase function is $-\varphi$ and the canonical relation $(\Lambda^\prime)^\ast$ of $A^\ast$ is given by 
$$
(\Lambda^\prime)^\ast
=
\{(y,x;-\varphi_y^\prime,\varphi_x^\prime) : \varphi^\prime_\theta=0\}
=
\{(y,x;\eta,\xi) : (x,y;\xi,\eta)\in\Lambda^\prime\}. 
$$
When we consider the action of Fourier integral operators on Lagrangian distributions, 
we need to use so-called the intersection calculus of Lagrangian submanifolds. 
Here we introduce some basic notions of the intersection calculus, those are, transversal intersections and clean intersections of two submanifolds. See \cite{MelroseUhlmann} and \cite{GuilleminUhlmann} for the detail. 
\begin{definition}
\label{theorem:transversality}
Let $X$ be a smooth manifold, and let $Y$ and $Z$ be submanifolds of $X$. 
\begin{itemize}
\item 
We say that $Y$ and $Z$ intersect transversely if $N^\ast_xY{\cap}N^\ast_xZ=\{0\}$ for all $x \in Y{\cap}Z$. Note that this condition is equivalent to that $T_xY{\cup}T_xZ=T_xX$ for all $x \in Y{\cap}Z$. 
\item 
We say that $Y$ and $Z$ intersect cleanly if $Y{\cap}Z$ is smooth and $T_xY{\cap}T_xZ=T_x(Y{\cap}Z)$ for all $x \in Y{\cap}Z$. Moreover, 
$$
e:=
\operatorname{codim}(Y)+\operatorname{codim}(Z)-\operatorname{codim}(Y{\cap}Z)
$$
is said to be the excess of the intersection. 
\end{itemize}
\end{definition}
Note that transversal intersection is clean intersection with $e=0$. 
In this paper we need to consider compositions of Lagrangian submanifolds using the intersection calculus. Following \cite[Chapter~21]{Hoermander3}, we introduce the clean composition. 
\begin{definition}
\label{theorem:cleanlangangian}
Let $X$ and $Y$ be smooth manifolds, 
and let $C \subset T^\ast(X{\times}Y)$ and $C_1 \subset T^\ast{Y}$ 
be conic Lagrangian submanifolds. 
We say that $C{\circ}C_1$ is clean with excess $e$ 
if $C{\times}C_1$ intersects $T^\ast{X}\times\Delta(T^\ast{Y})$ 
cleanly with excess $e$ in $T^\ast(X{\times}Y{\times}Y)$, 
where $\Delta(T^\ast{Y})$ is the diagonal part of $T^\ast{Y}{\times}T^\ast{Y}$. 
In particular, we say that $C{\circ}C_1$ is transversal 
if $C{\circ}C_1$ is clean with $e=0$. 
\end{definition}
More concretely, if we set $\dim{X}=n_X$ and $\dim{Y}=n_Y$, 
we say that $C{\circ}C_1$ is clean with excess $e$ when   
$$
M
:=
\bigl\{
\bigl((x,\xi),(y,\eta),(y^\prime,\eta^\prime)\bigr)\in C{\times}C_1 : 
(y,\eta)=(y^\prime,\eta^\prime)
\bigr\}
$$
is a smooth manifold and $\dim{M}=n_X+e$. 
We now state the well-known results \cite[Theorem~25.2.3]{Hoermander4} for the actions of Fourier integral operators on Lagrangian distributions. 
\begin{lemma}
\label{theorem:hoermander2523} 
Let $X$ and $Y$ be smooth manifolds, 
and let $C \subset T^\ast(X{\times}Y)$ and $C_1 \subset T^\ast{Y}$ 
be conic Lagrangian submanifolds. 
Suppose that $A$ is a Fourier integral operator whose distribution kernel belongs to 
$I^m(X{\times}Y,C^\prime)$ and $u \in I^{m_1}(Y,C_1)$. 
If $C{\circ}C_1$ is clean with excess $e$, then 
$Au \in I^{m+m_1+e/2}(X,C{\circ}C_1)$.  
\end{lemma}
%
%
\section{Canonical relations of the $d$-plane transform}
\label{section:dplane} 
In this section we present the basic properties of the $d$-plane transform as a Fourier integral operator. More precisely we show that the distribution kernel of the $d$-plane transform is a Lagrangian distribution on $G(d,n)\times\mathbb{R}^n$ with some order and some conic Lagrangian submanifold of $T^\ast\bigl(G(d,n)\times\mathbb{R}^n\bigr)\setminus0$. So we begin with the expression of $T^\ast{G(d,n)}$. We should understand the structure of the mapping \eqref{equation:canonicaltransformation} for $Y=\mathbb{R}^n$ and $X=G(d,n)$ concretely and clearly so that we can make full use of microlocal analysis. This depends on the expression of $T^\ast{G(d,n)}$. Here we recall that $\operatorname{dim}G_{d,n}=d(n-d)$ and $N(d,n):=\operatorname{dim}G(d,n)=(d+1)(n-d)$. 
\begin{lemma}[The expression of $T^\ast{G(d,n)}$]
\label{theorem:cotangentbundle}
$$
T^\ast{G(d,n)}
=
\{
(\sigma,x^{\prime\prime};\eta_1,\dotsc,\eta_d,\xi) 
:   
\sigma{\in}G_{d,n}, 
\ 
x^{\prime\prime}, 
\eta_1, 
\dotsc, 
\eta_d, 
\xi 
\in
\sigma^\perp
\}.
$$
\end{lemma}
\begin{proof}[{\bf Proof}]
Fix arbitrary $(\sigma,x^{\prime\prime}) \in G(d,n)$. 
It is easy to see that the set of cotangent vectors corresponding to $x^{\prime\prime}$ is $\xi\in\sigma^\perp$. There exists an orthonormal system $\{\omega_1,\dotsc,\omega_d\}$ of $\mathbb{R}^n$ such that $\sigma=\langle\omega_1,\dotsc,\omega_d\rangle$. We remark that the covariant derivative of a function on $\sigma$ with respect to $\omega_j$ is a vector in $\omega_j^\perp$. Recall that the set of cotangent vectors corresponding to $\sigma$ is a $d(n-d)$-dimensional real vector space. If we replace $\omega_j$ by $c\omega_j$ with $c\ne0$, then 
$$
\langle\omega_1,\dotsc,\omega_{j-1},c\omega_j,\omega_{j+1},\dotsc,\omega_d\rangle
=
\langle\omega_1,\dotsc,\omega_d\rangle,
\quad
j=1,\dotsc,d.
$$
This means that any cotangent vector corresponding $\sigma$ must belong to $\omega_1^\perp\cap\dotsb\cap\omega_d^\perp=\sigma^\perp$. Hence the vector space of all cotangent vectors corresponding $\sigma$ can be seen as the $d$-product $\sigma^\perp\times\dotsb\times\sigma^\perp$.  
\end{proof}
We denote by $\pi_\sigma$ the orthogonal projection of $\mathbb{R}^n$ onto $\sigma \in G_{d,n}$. 
If $\sigma$ is spanned by an orthonormal system 
$\{\omega_1,\dotsc,\omega_d\}$ of $\mathbb{R}^n$, 
then $\pi_\sigma$ is given by 
$$
\pi_\sigma=\sum_{j=1}^d\omega_j\omega_j^T. 
$$
We denote this case by $\sigma=\langle\omega_1,\dotsc,\omega_d\rangle_{\operatorname{ON}}$. 
This notation depends on the choice of an orthonormal basis of $\sigma$, 
and we need to take care of it below. 
We have the quantitative properties of $\mathcal{R}_d$ as follows. 
\begin{theorem}
\label{theorem:canonicalrd}
$\mathcal{R}_d$ is an elliptic Fourier integral operator whose distribution kernel belongs to 
$$
I^{0-(N(d,n)+n)/4+(n-d)/2}\bigl(G(d,n)\times\mathbb{R}^n,\Lambda_\phi\bigr),
$$
$$
0-\frac{N(d,n)+n}{4}+\frac{n-d}{2}
=
-
\frac{d(n-d+1)}{4},
$$
\begin{align*}
  \Lambda_\phi^\prime
& =
  \bigl\{
  \bigl(
  \sigma,y-\pi_\sigma{y},y;
  \eta(y\cdot\omega_1,\dotsc,y\cdot\omega_d,1,1)
  \bigr) 
  :
\\
& \qquad\qquad 
  \sigma=\langle\omega_1,\dotsc,\omega_d\rangle_{\operatorname{ON}}{\in}G_{d,n}, 
  y\in\mathbb{R}^n, \eta\in\sigma^\perp\setminus\{0\}
  \bigr\} 
\\
& =
  \bigl\{
  \bigl(
  \sigma,y-{\pi_\sigma}y,y;
  \eta(y\cdot\omega_1,\dotsc,y\cdot\omega_d,1,1)
  \bigr)
  : 
\\
& \qquad\qquad
  (y,\eta) \in T^\ast\mathbb{R}^n\setminus0, 
  \sigma=\langle\omega_1,\dotsc,\omega_d\rangle_{\operatorname{ON}}\in G_{d,n}, 
  \sigma\subset\eta^\perp
  \bigr\}
\\
& =
  \bigl\{
  \bigl(
  \sigma,x^{\prime\prime},x^{\prime\prime}+t_1\omega_1+\dotsb+t_d\omega_d;
  \xi(t_1,\dotsc,t_d,1,1)
  \bigr)
\\
& \qquad\qquad 
  : 
  (\sigma,x^{\prime\prime}){\in}G(d,n), 
  \sigma=\langle\omega_1,\dotsc,\omega_d\rangle_{\operatorname{ON}}\in G_{d,n}, 
  t_1,\dotsc,t_d\in\mathbb{R}, 
  \xi\in\sigma^\perp\setminus\{0\}
  \bigr\}.
\end{align*} 
\end{theorem}
\begin{proof}[{\bf Proof}]
We begin with the generalized Fourier slice theorem. For any $\sigma{\in}G_{d,n}$ and $\xi\in\sigma^\perp$, we have $x^\prime\cdot\xi=0$ for $x^\prime\in\sigma$ and  
\begin{align*}
  \int_{\sigma^\perp}
  e^{-ix^{\prime\prime}\cdot\xi}
  \mathcal{R}_df(\sigma,x^{\prime\prime})
  dx^{\prime\prime}
& =
  \int_{\sigma^\perp}
  \int_\sigma
  e^{-ix^{\prime\prime}\cdot\xi}
  f(x^\prime+x^{\prime\prime})
  dx^\prime
  dx^{\prime\prime} 
\\
& =
  \int_{\sigma^\perp}
  \int_\sigma
  e^{-i(x^\prime+x^{\prime\prime})\cdot\xi}
  f(x^\prime+x^{\prime\prime})
  dx^\prime
  dx^{\prime\prime}
\\
& =
  \int_{\mathbb{R}^n}
  e^{-ix\cdot\xi}
  f(x)
  dx
  =
  \hat{f}(\xi),
  \quad
  \xi\in\sigma^\perp.  
\end{align*}
Then for any $(\sigma,x^{\prime\prime}){\in}G(d,n)$, the Fourier inversion formula on $\sigma^\perp$ implies that 
\begin{align*}
  \mathcal{R}_df(\sigma,x^{\prime\prime})
& =
  \frac{1}{(2\pi)^{n-d}}
  \int_{\sigma^\perp}
  e^{ix^{\prime\prime}\cdot\xi}
  \hat{f}(\xi)
  d\xi
\\
& =
  \frac{1}{(2\pi)^{n-d}}
  \int_{\sigma^\perp}
  \int_{\mathbb{R}^n}
  e^{i(x^{\prime\prime}-y)\cdot\xi}
  f(y)
  dy
  d\xi
\\
& =
  \frac{1}{(2\pi)^{n-d}}
  \int_{\sigma^\perp}
  \int_{\mathbb{R}^n}
  e^{i\phi(\sigma,x^{\prime\prime},y,\xi)}
  f(y)
  dy
  d\xi,
\\
  \phi(\sigma,x^{\prime\prime},y,\xi)
& =
  (x^{\prime\prime}-y)\cdot\xi, 
  \quad
  (\sigma,x^{\prime\prime}){\in}G(d,n), 
  y\in\mathbb{R}^n,
  \xi\in\sigma^\perp. 
\end{align*}
It suffices to show that $\phi$ is a non-degenerate phase function, and to obtain the Lagrangian submanifold $\Lambda_\phi$. If this is true, then it follows that the amplitude is the function which is identically equal to $1/(2\pi)^{n-d}$. This symbol belongs to $S^0\Bigl(\bigl(G(d,n)\times\mathbb{R}^n\bigr)\times\mathbb{R}^{n-d}\Bigr)$ and the order of the Lagrangian distribution is 
$$
0+\frac{N(d,n)+n}{4}-\frac{n-d}{2}. 
$$
\par
We find the critical points defined by $\phi_\xi^\prime=0$ for $\xi\in\sigma^\perp$. 
Set
$$
\psi_1(\sigma,x^{\prime\prime},y,\Xi)
:=
\phi(\sigma,x^{\prime\prime},y,\Xi-\pi_\sigma\Xi)
=
(x^{\prime\prime}-y)\cdot(\Xi-\pi_\sigma\Xi)
=
(x^{\prime\prime}-y+\pi_\sigma{y})\cdot\Xi
$$
for $\Xi\in\mathbb{R}^n$. Then $\Xi$ moves essentially in $\sigma^\perp$ and not in the whole space $\mathbb{R}^n$ for $\psi_1$. Then we have 
$$
\phi_\xi^\prime
=
\nabla_\Xi\psi_1\vert_{\Xi=\xi}
=
x^{\prime\prime}-(y-\pi_\sigma{y}). 
$$
Hence the critical points of $\phi$ are characterized by $x^{\prime\prime}=y-\pi_\sigma{y}$. 
Let $E$ be the $n{\times}n$ identity matrix. 
We check the non-degeneracy of the critical points. We have 
$$
\phi^{\prime\prime}_{\xi y}=-E+\pi_\sigma, 
\quad
\operatorname{rank}\phi^{\prime\prime}_{\xi y}=n-d. 
$$
Then we have $\operatorname{rank}(d\phi^\prime_\xi)\vert_{x=y-\pi_\sigma{y}}=n-d$ since $\xi$ moves in the $(n-d)$-dimensional vector subspace $\sigma^\perp$. Finally we obtain $\phi^\prime_{\sigma,x^{\prime\prime},y,\xi}\ne0$ for $\xi\in\sigma^\perp\setminus\{0\}$ since $\phi^\prime_y=-\xi$. Hence we have proved that $\phi(\sigma,x^{\prime\prime},y,\xi)$ is a non-degenerate phase function with critical points given by $x^{\prime\prime}=y-\pi_\sigma{y}$. 
\par
Finally we compute the Lagrangian submanifold 
$$
\Lambda_\phi
=
\{
(\sigma,x^{\prime\prime},y ; \phi^\prime_\sigma,\phi^\prime_{x^{\prime\prime}},\phi^\prime_y) 
: 
(\sigma,x^{\prime\prime}) \in G(d,n), y\in\mathbb{R}^n, \xi\in\sigma^\perp, x^{\prime\prime}=y-\pi_\sigma{y}
\}
$$
of $T^\ast\bigl(G(d,n)\times\mathbb{R}^n\bigr)\setminus0$. 
Recall that $\phi^\prime_y=-\xi$. 
Set 
$$
\psi_2(\sigma,x,y,\xi)
:=
\phi(\sigma,x-\pi_\sigma{x},y,\xi)
=
(x-y)\cdot\xi
$$
for $(\sigma,x,y,\xi) \in G_{d,n}\times\mathbb{R}^n\times\mathbb{R}^n\times\sigma^\perp$. 
Then we have $\phi_{x^{\prime\prime}}^\prime=\nabla_x\psi_2\vert_{x=x^{\prime\prime}}=\xi$. 
Next we compute $\phi^\prime_\sigma$. Suppose that $\sigma$ is spanned by an orthonormal system 
$\{\omega_1,\dotsc,\omega_d\}$ in $\mathbb{R}^n$. 
We compute $\phi^\prime_{\omega_j}$. 
Consider a small perturbation of $\sigma$ of the form 
$$
\tilde{\sigma}
=
\langle\mu_1,\dotsc.\mu_d\rangle=\sigma+o(1),
\quad
\mu_1,\dotsc.\mu_d\in\mathbb{R}^n. 
$$
Set 
$$
\psi_3(\tilde{\sigma},x^{\prime\prime},y,\xi)
:=
\left(
x-y+\sum_{j=1}^d(y\cdot\mu_j)\mu_j
\right)
\cdot\xi
=
(x^{\prime\prime}-y)\cdot\xi
+
\sum_{j=1}^d(y\cdot\mu_j)(\xi\cdot\mu_j). 
$$
Then we have 
$$
\phi^\prime_{\omega_j}
=
\bigl(
\nabla_{\mu_j}\psi_3
-
(\mu_j\cdot\nabla_{\mu_j}\psi_3)\mu_j
\bigr)\vert_{\mu_j=\omega_j}
=
(\xi\cdot\omega_j)y+(y\cdot\omega_j)\xi
=
(y\cdot\omega_j)\xi
$$
since $\xi\in\langle\omega_1,\dotsc,\omega_d\rangle^\perp$. 
Hence we obtain 
$$
\phi^\prime_{\sigma,x^{\prime\prime},y}
=
\xi(y\cdot\omega_1,\dotsc,y\cdot\omega_d,1,-1)
$$
for $x^{\prime\prime} \in \sigma^\perp$, 
$y\in\mathbb{R}^n$ 
and 
$\xi \in \sigma^\perp\setminus\{0\}$. 
Summing up the above computation, we deduce that 
\begin{align*}
  \Lambda_\phi
& =
  \bigl\{
  \bigl(
  \sigma,y-\pi_\sigma{y},y;
  \eta(y\cdot\omega_1,\dotsc,y\cdot\omega_d,1,-1)
  \bigr) 
  :
\\
& \qquad\qquad 
  \sigma=\langle\omega_1,\dotsc,\omega_d\rangle_{\operatorname{ON}}{\in}G_{d,n}, 
  y\in\mathbb{R}^n, \eta\in\sigma^\perp\setminus\{0\}
  \bigr\} 
\\
& =
  \bigl\{
  \bigl(
  \sigma,y-{\pi_\sigma}y,y;
  \eta(y\cdot\omega_1,\dotsc,y\cdot\omega_d,1,-1)
  \bigr)
  : 
\\
& \qquad\qquad
  (y,\eta) \in T^\ast\mathbb{R}^n\setminus0, 
  \sigma=\langle\omega_1,\dotsc,\omega_d\rangle_{\operatorname{ON}}\in G_{d,n}, 
  \sigma\subset\eta^\perp
  \bigr\}
\\
& =
  \bigl\{
  \bigl(
  \sigma,x^{\prime\prime},x^{\prime\prime}+t_1\omega_1+\dotsb+t_d\omega_d;
  \xi(t_1,\dotsc,t_d,1,-1)
  \bigr)
\\
& \qquad\qquad 
  : 
  (\sigma,x^{\prime\prime}){\in}G(d,n), 
  \sigma=\langle\omega_1,\dotsc,\omega_d\rangle_{\operatorname{ON}}\in G_{d,n}, 
  t_1,\dotsc,t_d\in\mathbb{R}, 
  \xi\in\sigma^\perp\setminus\{0\}
  \bigr\}.
\end{align*} 
This completes the proof.  
\end{proof}
%
%
%
%
\section{Geometry of metal regions}
\label{section:metalregion}
In Section~\ref{section:introduction} we reviewed the known results on the artifacts caused by a metal region $D$ in $\mathbb{R}^2$. In the present section we study the metal streaking artifacts in general space dimension $n$. Let $E\in[0,\infty)$ be the energy level of X-rays and let $E_0>0$ be the standard level for human tissue. Suppose that the distribution of attenuation coefficients on $\mathbb{R}^n$ depends on $E$ of the form 
$$
f_E(x)=f_{E_0}(x)+\alpha(E-E_0)\chi_D(x), 
\quad
x\in\mathbb{R}^n, 
$$
where $\alpha>0$ is a constant. The spectral function is supposed to be a probability density of uniformly distributed on a closed interval $[E_0-\varepsilon,E_0+\varepsilon]\subset[0,\infty)$ of the form 
$$
\rho(E)
=
\frac{1}{2\varepsilon}\chi_{[E_0-\varepsilon,E_0+\varepsilon]}(E),
\quad
E\in[0,\infty).
$$ 
Then the measurement is not $\mathcal{R}_df_E(\sigma,x^{\prime\prime})$ but 
\begin{align*}
  P_d(\sigma,x^{\prime\prime})
& =
  -
  \log
  \left\{
  \int_{E_0-\varepsilon}^{E_0+\varepsilon}
  \rho(E)
  \exp\bigl(-\mathcal{R}_df_E(\sigma,x^{\prime\prime})\bigr)
  dE
  \right\} 
\\
& =
  \mathcal{R}_df_{E_0}(\sigma,x^{\prime\prime})
  +
  P_{d,\text{MA}}(\sigma,x^{\prime\prime}),
\\
  P_{d,\text{MA}}(\sigma,x^{\prime\prime})
& =
  \sum_{k=1}^\infty
  \frac{(-1)^k}{k}
  \left\{
  \sum_{l=1}^\infty
  \frac{\bigl(\alpha\varepsilon\mathcal{R}_d\chi_D(\sigma,x^{\prime\prime})\bigr)^{2l}}{(2l+1)!}
  \right\}^k
\\
& =
  \sum_{l=1}^\infty
  A_l
  \bigl(\alpha\varepsilon\mathcal{R}_d\chi_D(\sigma,x^{\prime\prime})\bigr)^{2l}, 
\end{align*}
where $\{A_l\}_{l=1}^\infty$ is a sequence of real numbers. 
The CT image is denoted by $f_\text{CT}$, that is, 
\begin{align*}
  f_\text{CT}
& :=
  \mathcal{R}_d^\ast(-\Delta_{x^{\prime\prime}})^{d/2}P_d
  =
  f_{E_0}+f_\text{MA},
\\
  f_\text{MA}
& =
  \mathcal{R}_d^\ast(-\Delta_{x^{\prime\prime}})^{d/2}P_{d,\text{MA}}
  =
  \sum_{l=1}^\infty
  A_l
  (\alpha\varepsilon)^{2l}
   \mathcal{R}_d^\ast(-\Delta_{x^{\prime\prime}})^{d/2}
   \bigl[(\mathcal{R}_d\chi_D)^{2l}\bigr]
\end{align*}
The metal streaking artifacts are nonlinear effects and we need to study the microlocal singularities of 
$$
(\mathcal{R}_d\chi_D)^{2l}, 
\quad
\mathcal{R}_d^\ast(-\Delta_{x^{\prime\prime}})^{d/2}\bigl[(\mathcal{R}_d\chi_D)^{2l}\bigr]
\quad
(l=1,2,3,\dotsc). 
$$ 
Throughout the present paper we assume the following:
\begin{itemize}
\item[(A)] 
$D$ is a union of bounded strictly convex smooth domains 
$D_1,\dotsc,D_J \subset \mathbb{R}^n$ with $J=1,2,3,\dotsc$ 
such that $\overline{D_1},\dotsc,\overline{D_J}$ are mutually disjoint. 
\end{itemize}
Set $\Sigma_j:=\partial{D}_j$ and $\Sigma:=\displaystyle\bigcup_{j=1}^J\Sigma_j$. 
Then $\Sigma=\partial{D}$. 
We denote by $\pi_{G(d,n)}$ the projection of $T^\ast{G(d,n)}$ onto $G(d,n)$. 
We prepare some lemmas related with basic properties of 
$\Lambda_\phi^\prime{\circ}N^\ast\Sigma_j$. 
The first lemma says that $\Lambda_\phi^\prime{\circ}N^\ast\Sigma_j$ 
is also a conormal bundle of a hypersurface of $G(d,n)$. 
\begin{lemma}
\label{theorem:lemma1}
A composition $\Lambda_\phi^\prime{\circ}N^\ast\Sigma_j$ is transversal. 
Moreover if we set $S_j:=\pi_{G(d,n)}(\Lambda_\phi^\prime{\circ}N^\ast\Sigma_j)$, 
then 
$\operatorname{codim}S_j=1$ 
and 
$N^\ast{S_j}=\Lambda_\phi^\prime{\circ}N^\ast\Sigma_j$. 
\end{lemma}
\begin{proof}[{\bf Proof}]
Recall that $N^\ast\Sigma_j$ and $\Lambda_\phi^\prime$ are concretely given by 
$$
N^\ast\Sigma_j
=
\{
(y,\eta) \in T^\ast\mathbb{R}^n 
: 
y \in \Sigma_j, 
\eta \in N^\ast_y\Sigma_j
\},
$$
and 
\begin{align*}
  \Lambda_\phi^\prime
& =
  \bigl\{
  \bigl(
  \sigma,y-{\pi_\sigma}y,y;
  \eta(y\cdot\omega_1,\dotsc,y\cdot\omega_d,1,1)
  \bigr)
  : 
\\
& \qquad\qquad
  (y,\eta) \in T^\ast\mathbb{R}^n\setminus0, 
  \sigma=\langle\omega_1,\dotsc,\omega_d\rangle_{\operatorname{ON}}\in G_{d,n}, 
  \sigma\subset\eta^\perp
  \bigr\}
\end{align*}
respectively. Then we have 
\begin{align*}
  M_1
& :=
  \bigl(T^\ast{G(d,n)}\times\Delta(T^\ast\mathbb{R}^n)\bigr)\cap(\Lambda_\phi^\prime{\times}N^\ast\Sigma_j) 
\\
&  =
   \Bigl\{
  \Bigl(
  \bigl(\sigma,y-\pi_\sigma{y},\eta(\pi_\sigma{y},1)\bigr),(y,\eta), (y,\eta)
  \Bigr) : 
  (y,\eta) \in N^\ast\Sigma_j\setminus0, 
  \sigma \in G_{d,n}\cap\eta^\perp
  \Bigr\}
\end{align*}
and 
\begin{align*}
  \dim{M_1}
& =
  \dim(N^\ast\Sigma_j)+\dim(G_{d,n}\cap\eta^\perp)
  =
  \dim(N^\ast\Sigma_j)+\dim(G_{d,n-1})
\\
& =
  n+d(n-1-d)  
  =
  (d+1)(n-d)
  =
  N(d,n)+0. 
\end{align*}
Hence $\Lambda_\phi^\prime{\circ}N^\ast\Sigma_j$ is transversal. 
Moreover we have  
\begin{align*}
  \Lambda_\phi^\prime{\circ}N^\ast\Sigma_j
& =
  \bigl\{
  \bigl(
  \sigma,y-{\pi_\sigma}y;\eta(y\cdot\omega_1,\dotsc,y\cdot\omega_d,1)
  \bigr)
  : 
\\
& \qquad\qquad
  (y,\eta){\in}N^\ast\Sigma_j\setminus0, 
  \sigma=\langle\omega_1,\dotsc,\omega_d\rangle_{\operatorname{ON}}\in G_{d,n}, 
  \sigma\subset\eta^\perp
  \bigr\} 
\end{align*}
and therefore 
$$
S_j
=
\pi_{G(d,n)}(\Lambda_\phi^\prime{\circ}N^\ast\Sigma_j)
=
\{
(\sigma,y-{\pi_\sigma}y)
 : 
y\in\Sigma_j, 
\sigma \in G_{d,n}, 
\sigma \subset T_y\Sigma_j
\}. 
$$
Then we deduce that 
\begin{align*}
  \operatorname{dim}S_j
& =
  \operatorname{dim}\Sigma_j
  +
  \operatorname{dim}G_{d,n-1}
  =
  (n-1)+d(n-1-d)
\\
& =
  (d+1)(n-d)-1
  =
  N(d,n)-1. 
\end{align*}
This shows that $S_j$ is a hypersurface in $G(d,n)$ and $\operatorname{codim}S_j=1$. 
The canonical symplectic form of $T^\ast\mathbb{R}^n$ vanishes on $N^\ast\Sigma_j$, 
and the canonical transform preserves the canonical symplectic form. 
Hence the canonical symplectic form of $T^\ast{G(d,n)}$ also vanishes on 
$\Lambda_\phi^\prime{\circ}N^\ast\Sigma_j$. 
Combining these facts we deduce that $N^\ast{S_j}=\Lambda_\phi^\prime{\circ}N^\ast\Sigma_j$ 
since $N^\ast{S_j}$ and $\Lambda_\phi^\prime{\circ}N^\ast\Sigma_j$ 
are conic Lagrangian submanifolds with the same base space $S_j$. 
\end{proof}
Next we discuss the intersection $S_{jk}:=S_j{\cap}S_k$ for $j{\ne}k$. The structure of the intersection of Lagrangian submanifolds determines some properties of the product of Lagrangian distributions. So the intersection calculus of $S_{jk}$ for $j{\ne}k$ plays a crucial role in the microlocal analysis of nonlinear effects of $\mathcal{R}_d\chi_D$. In particular we need to know $\operatorname{codim}S_{jk}$ and the relationship between $N^\ast{S_j}$, $N^\ast{S_k}$ and $N^\ast{S_{jk}}$. We begin with some basic properties of  the intersection $S_{jk}$ for $j{\ne}k$. We first remark that  $S_{jk}\ne\emptyset$ holds if and only if there exist 
$(\sigma,x^{\prime\prime}) \in S_{jk}$, $y_j\in\Sigma_j$ and $y_k\in\Sigma_k$ such that 
$y_j{\ne}y_k$ since $\Sigma_j\cap\Sigma_k=\emptyset$, and 
$$
x^{\prime\prime}=y_j-\pi_\sigma{y_j}=y_k-\pi_\sigma{y_k}, 
\quad
\sigma \subset T_{y_j}\Sigma_j{\cap}T_{y_k}\Sigma_k. 
$$
If $N^\ast_{y_j}\Sigma=N^\ast_{y_k}\Sigma$, then $T_{y_j}\Sigma_j=T_{y_k}\Sigma_k$, 
otherwise 
$$
T_{y_j}\Sigma_j{\cap}T_{y_k}\Sigma_k
=
\bigl(N^\ast_{y_j}\Sigma_j\bigr)^\perp\cap\bigl(N^\ast_{y_k}\Sigma_k\bigr)^\perp
$$ 
is an $(n-2)$-dimensional vector subspace. See Figures~2 and 3 below. 
\begin{multicols}{2}
\begin{center}
\includegraphics[width=50mm]{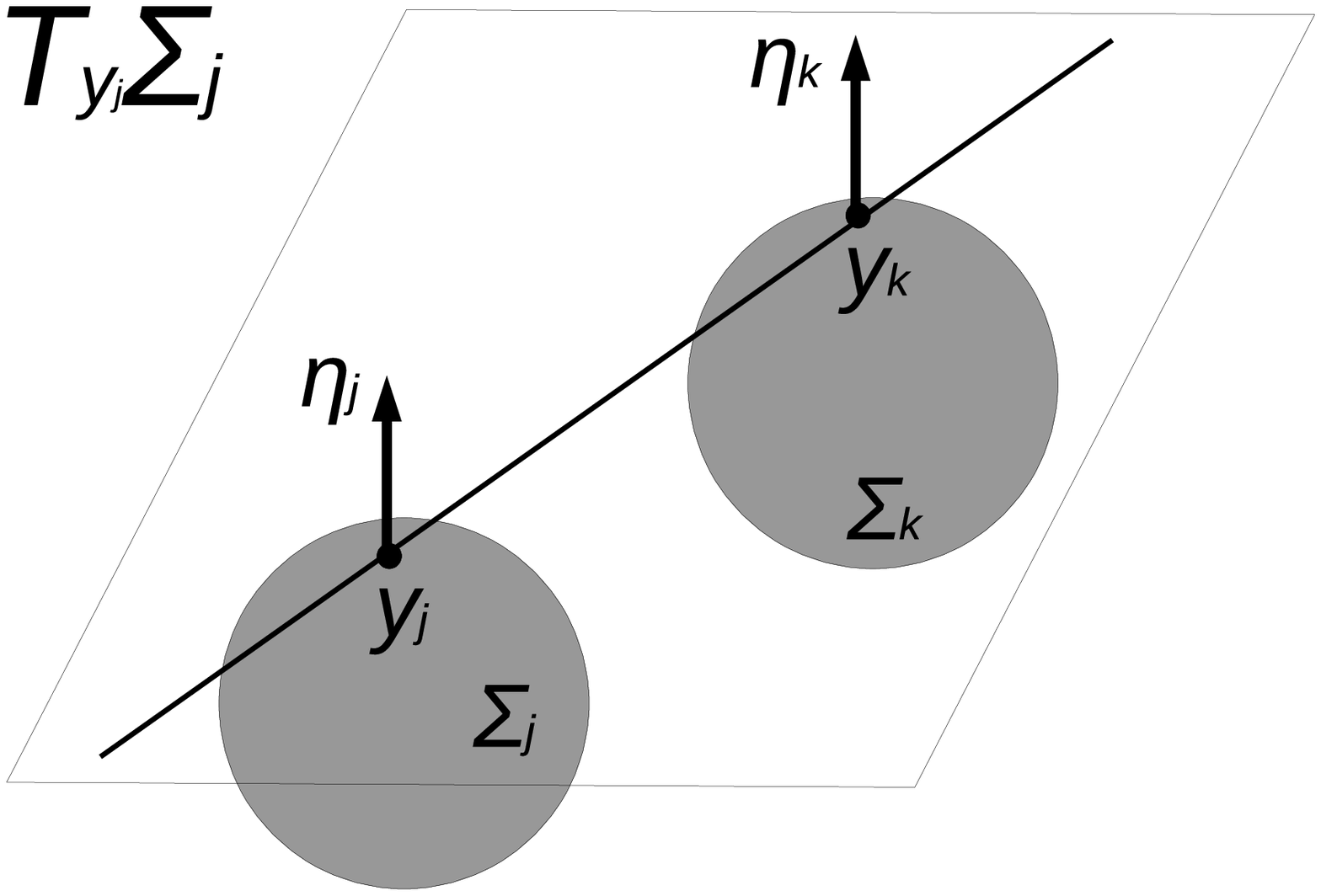}
\\
Figure~2. 
This illustrates the case of $N^\ast_{y_j}\Sigma=N^\ast_{y_k}\Sigma$ with 
$\eta_j, \eta_k \in N^\ast\Sigma_j\setminus\{0\}=N^\ast\Sigma_k\setminus\{0\}$. 
\end{center}
\begin{center}
\includegraphics[width=50mm]{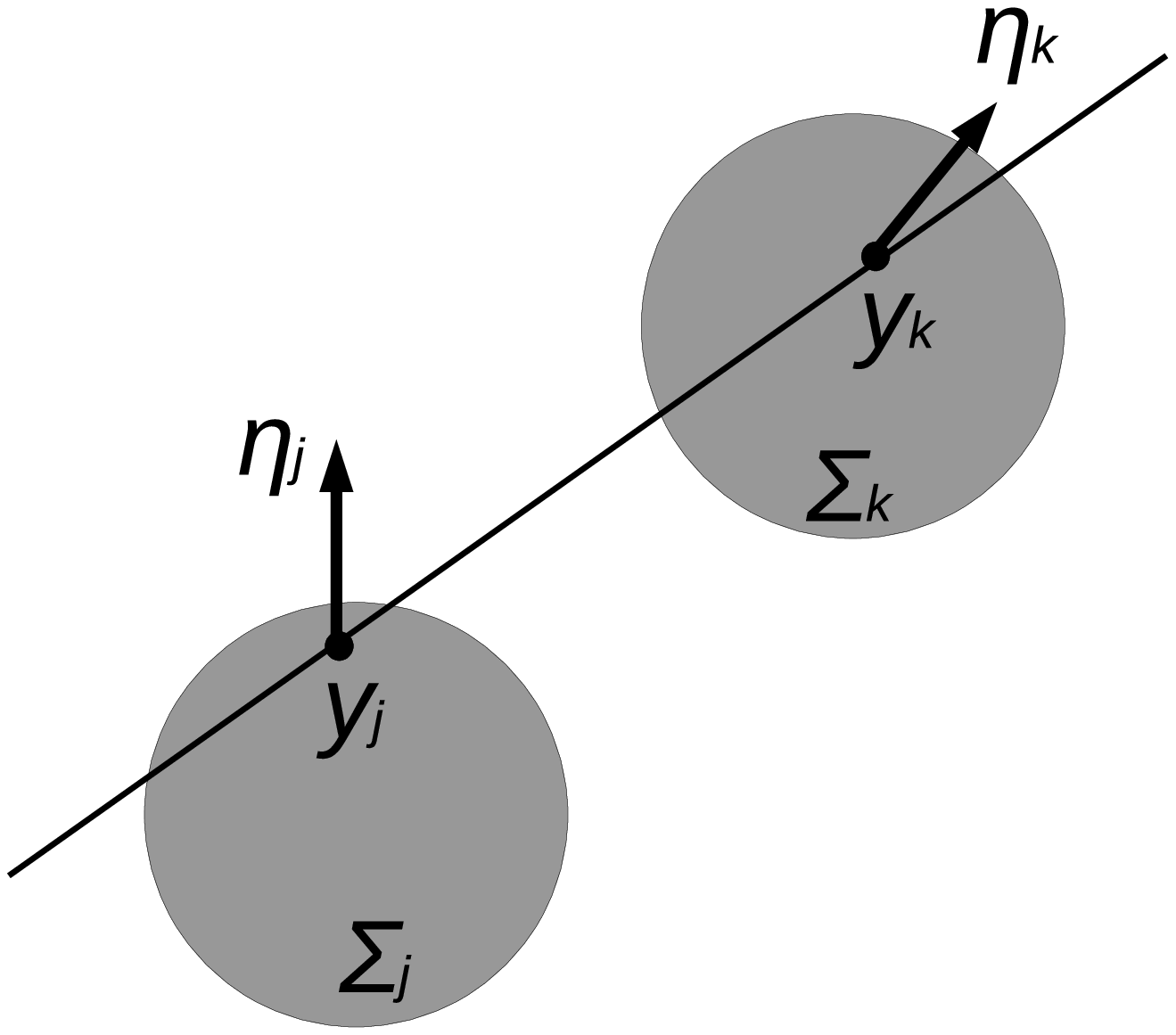}
\\
Figure~3. 
This illustrates the case of $N^\ast_{y_j}\Sigma{\ne}N^\ast_{y_k}\Sigma$ with 
$\eta_j \in N^\ast\Sigma_j\setminus\{0\}$ and $\eta_k \in N^\ast\Sigma_k\setminus\{0\}$. 
\end{center}
\end{multicols}
We split ${S_{jk}}$ into two parts: 
\begin{align*}
  S_{jk}^{(1)}
& = 
  \{(\sigma,x^{\prime\prime}) \in S_{jk} : N^\ast_{y_j}\Sigma=N^\ast_{y_k}\Sigma\},
\\
  S_{jk}^{(2)}
& = 
  \{(\sigma,x^{\prime\prime}) \in S_{jk} : N^\ast_{y_j}\Sigma{\ne}N^\ast_{y_k}\Sigma\}. 
\end{align*}
Clearly we have a disjoint union $S_{jk}=S_{jk}^{(1)}{\cup}S_{jk}^{(2)}$, 
and in particular $S_{jk}=S_{jk}^{(1)}$ and $S_{jk}^{(2)}=\emptyset$ if $d=n-1$. 
Here we state the properties of the intersection $S_{jk}$. 
\begin{lemma}
\label{theorem:intersection}
Suppose that $j{\ne}k$ and $S_{jk}\ne\emptyset$. 
Then $S_j$ and $S_k$ intersect transversely. 
\end{lemma}
\begin{proof}[{\bf Proof}]
Fix arbitrary $(\sigma,x^{\prime\prime}) \in S_{jk}$. 
There exist $y_j\in\Sigma_j$ and $y_k\in\Sigma_k$ such that 
$\sigma \subset T_{y_j}\Sigma_j{\cap}T_{y_k}\Sigma_k$. 
Hence we can choose an orthonormal system $\{\omega_0,\omega_1\}\subset\sigma$ 
and $t_0,t_j,t_k\in\mathbb{R}$ such that 
\begin{equation}
y_j=x^{\prime\prime}+t_0\omega_0+t_j\omega_1,
\quad
y_k=x^{\prime\prime}+t_0\omega_0+t_k\omega_1.
\label{equation:t0}
\end{equation}
More precisely $\omega_1$ is a vector of the direction of $y_k-y_j$ and 
we can set $\omega_1:=(y_k-y_j)/\lvert{y_k-y_j\rvert}$. 
$\omega_0$ is a perpendicular direction from $x^{\prime\prime}$ 
to the line connecting $y_j$ and $y_k$. 
In case $d=1$ we think that $t_0=0$. 
We have $t_j{\ne}t_k$ since $\Sigma_j\cap\Sigma_k=\emptyset$. 
Recall Lemma~\ref{theorem:cotangentbundle}. 
In the same way as the proof of Lemma~\ref{theorem:lemma1}, we have 
\begin{align*}
  N^\ast_{(\sigma,x^{\prime\prime})}{S_j}
& =
  \{\eta_j(t_0,t_j,0,\dotsc,0,1) : \eta_j \in N^\ast\Sigma_j\},
\\
  N^\ast_{(\sigma,x^{\prime\prime})}{S_k}
& =
  \{\eta_k(t_0,t_k,0,\dotsc,0,1) : \eta_k \in N^\ast\Sigma_k\}.  
\end{align*}
If $\eta_j\ne0$ and $\eta_k\ne0$, then vectors $\eta_j(t_0,t_j,0,\dotsc,0,1)$ and $\eta_k(t_0,t_k,0,\dotsc,0,1)$ are linearly independent in $T^\ast_{(\sigma,x^{\prime\prime})}G(d,n)$ since $t_j{\ne}t_k$, and therefore $N^\ast_{(\sigma,x^{\prime\prime})}{S_j}{\cap}N^\ast_{(\sigma,x^{\prime\prime})}{S_k}=\{0\}$. This completes the proof. 
\end{proof}
Lemma~\ref{theorem:intersection} implies that $\operatorname{codim}(S_{jk})=2$. Next we obtain the concrete expression of $N^\ast{S_{jk}}$ which is required for computing the composition 
$(\Lambda_\phi^\prime){\circ}N^\ast{S_{jk}}$. 
\begin{lemma}
\label{theorem:conormal2}
\begin{align*}
  N^\ast{S_{jk}^{(1)}}
& =
  \bigl\{
  \bigl(
  \sigma,x^{\prime\prime};\eta(t_0,t,0,\dotsc,0,1)
  \bigr)
  : 
  N^\ast_{y_j}\Sigma_j= N^\ast_{y_k}\Sigma_k, 
\\
& \qquad\qquad
  \sigma \subset T_{y_j}\Sigma_j, 
  x^{\prime\prime}=y_j-\pi_\sigma{y_j}=y_k-\pi_\sigma{y_k}, 
\\
& \qquad\qquad
  t\in\mathbb{R}, 
  \eta \in N^\ast_{y_j}\Sigma_j
  \ \text{with some}\ 
  t_0\in\mathbb{R}, y_j\in\Sigma_j, y_k\in\Sigma_k  
  \bigr\},
\\
  N^\ast{S_{jk}^{(2)}}
& = 
  \bigl\{
  \bigl(
  \sigma,x^{\prime\prime};\Xi
  \bigr)
  :
  N^\ast_{y_j}\Sigma_j \ne N^\ast_{y_k}\Sigma_k
\\
& \qquad\qquad
  \sigma \subset T_{y_j}\Sigma_j{\cap}T_{y_k}\Sigma_k, 
  x^{\prime\prime}=y_j-\pi_\sigma{y_j}=y_k-\pi_\sigma{y_k}, 
\\
& \qquad\qquad
  \Xi \in N^\ast_{(\sigma,x^{\prime\prime})}S_j \oplus N^\ast_{(\sigma,x^{\prime\prime})}S_k 
  \ \text{with some}\ 
  y_j\in\Sigma_j, y_k\in\Sigma_k  
  \bigr\}. 
\end{align*}
\end{lemma}
\begin{proof}[{\bf Proof}]
The statement on $N^\ast{S_{jk}^{(2)}}$ is obvious. 
We shall obtain $N^\ast{S_{jk}^{(1)}}$.  
Fix arbitrary $(\sigma,x^{\prime\prime}) \in S_{jk}^{(1)}$. 
We use the same notation in the proof of Lemma~\ref{theorem:intersection}. 
In view of Lemma~\ref{theorem:intersection} and its proof, we deduce that 
$$
N^\ast_{(\sigma,x^{\prime\prime})}S_{jk}^{(1)}
=
\langle\eta(t_0,t_j,0,\dotsc,0,1),\eta(t_0,t_k,0,\dotsc,0,1)\rangle
$$
with some $\eta \in N^\ast_{y_j}\Sigma_j\setminus\{0\}$. 
Recall that $t_j{\ne}t_k$. 
For any $t\in\mathbb{R}$, set $\alpha:=t_k-t$ and $\beta:=t-t_j$. 
Then $\alpha+\beta=t_k-t_j\ne0$, $t=(\alpha{t_j}+\beta{t_k})/(\alpha+\beta)$ and 
\begin{align*}
  \eta(t_0,t,0,\dotsc,0,1)
& =
  \eta\left(t_0,\frac{\alpha{t_j}+\beta{t_k}}{\alpha+\beta},0,\dotsc,0,1\right)
\\
& =
  \frac{\alpha}{\alpha+\beta}\eta(t_0,t_j,0,\dotsc,0,1)
  +
  \frac{\beta}{\alpha+\beta}\eta(t_0,t_k,0,\dotsc,0,1). 
\end{align*}
Hence we get 
$\{\eta(t_0,t,0,\dotsc,0,1) : t\in\mathbb{R}, \eta \in N^\ast_{y_j}\Sigma_j\} \subset N^\ast_{(\sigma,x^{\prime\prime})}S_{jk}$. The converse inclusion relation is obvious and we obtain 
$\{\eta(t_0,t,0,\dotsc,0,1) : t\in\mathbb{R}, \eta \in N^\ast_{y_j}\Sigma_j\}=N^\ast_{(\sigma,x^{\prime\prime})}S_{jk}$. This completes the proof. 
\end{proof}
We show that compositions 
$(\Lambda_\phi^\prime)^\ast{\circ}N^\ast{S_j}$, 
$(\Lambda_\phi^\prime)^\ast{\circ}N^\ast{S_{jk}^{(1)}}$ 
and 
$(\Lambda_\phi^\prime)^\ast{\circ}N^\ast{S_{jk}^{(2)}}$ 
are clean. 
We need this fact to consider the action of $\mathcal{R}_d^\ast$. 
For this purpose we study the relationship between $\Sigma_j$ and $\Sigma_k$. 
We denote by $\mathbb{S}^{n-1}$ the $(n-1)$-dimensional unit sphere with center at the origin in $\mathbb{R}^n$, 
and by $\nu_j(y)$ the unit outer normal vector of $\Sigma_j$ at $y\in\Sigma_j$. 
Set $\Pi_j(y)=\{x\in\mathbb{R}^n : \nu(y)\cdot(x-y)=0\}$, 
which is the tangent hyperplane of $\Sigma_j$ passing through $y\in\Sigma_j$. 
Since $D_1,\dotsc,D_J$ are strictly convex bounded smooth domains and their closures are disjoint, 
we have the following. 
\begin{itemize}
\item 
The map $\Sigma_j \ni y \mapsto \nu_j(y) \in \mathbb{S}^{n-1}$ is bijective. 
\item 
For any $y_j \in \Sigma_j$ there exit $y_k^{\pm} \in \Sigma_k$ uniquely such that 
$\nu_j(y_j)=\nu_k(y_k^+)=-\nu_k(y_k^-)$.  
\end{itemize} 
Considering these facts we introduce subsets of $\Sigma_j\times\Sigma_k$ 
describing the relationship between $\Sigma_j$ and $\Sigma_k$ for $j{\ne}k$. 
Set 
\begin{align*}
  \mathcal{M}_{jk}^{(\pm)}
& :=
  \{
  (y_j,y_k) \in \Sigma_j\times\Sigma_k : 
  \nu_j(y_j)\cdot(y_k-y_j)=0, 
\\
& \qquad 
  \nu_j(y_j)=\pm\nu_k(y_k)\ 
  \text{with some}\ y_k\in\Sigma_k
  \}, 
\end{align*}
which is the set of all pairs $(y_j,y_k) \in \Sigma_j\times\Sigma_k$ such that 
$\Pi_j(y_j)=\Pi_k(y_k)$ and $\nu_j(y_j)=\pm\nu_k(y_k)$ respectively. 
It is obvious that 
$$
\mathcal{M}_{kj}^{(\pm)}
=
\{
(y_k,y_j) \in \Sigma_k\times\Sigma_j : 
(y_j,y_k) \in \mathcal{M}_{jk}^{(\pm)}
\}.   
$$
Set $\mathcal{M}_{jk}:=\mathcal{M}_{jk}^{(+)}\cup\mathcal{M}_{jk}^{(-)}$ for short. 
Moreover we denote by $B_{jk}^{(\pm)}$ the projection of 
$\mathcal{M}_{jk}^{(\pm)}$ to $\Sigma_j$, that is,  
$$
B_{jk}^{(\pm)}
=
\{
y_j \in \Sigma_j : 
(y_j,y_k) \in \mathcal{M}_{jk}^{(\pm)},\ 
\text{}\ y_k \in \Sigma_k
\}. 
$$
\begin{lemma}
\label{theorem:omega}
\quad
\begin{itemize}
\item 
$\mathcal{M}_{jk}^{(\pm)}$ are $(n-2)$-dimensional connected submanifolds of 
$\Sigma_j\times\Sigma_k$ respectively, and 
$\mathcal{M}_{jk}^{(+)}\cap\mathcal{M}_{jk}^{(-)}=\emptyset$. 
\item 
The projections of $\mathcal{M}_{jk}^{(\pm)}$ onto $B_{jk}^{(\pm)}$ 
are diffeomorphic respectively. 
So $B_{jk}^{(\pm)}$ are $(n-2)$-dimensional connected submanifolds of 
$\Sigma_j$ respectively, and  
$B_{jk}^{(+)}{\cap}B_{jk}^{(-)}=\emptyset$.  
\item 
If we set 
$$
\mathcal{L}_{jk}^{(\pm)}
:=
\{
y_j+t(y_k-y_j) : 
(y_j,y_k)\in\mathcal{M}_{jk}^{(\pm)}, t\in\mathbb{R}
\}, 
$$
then $\mathcal{L}_{jk}^{(\pm)}$ are hypersurfaces in $\mathbb{R}^n$ respectively. 
Clearly $\mathcal{L}_{jk}^{(\pm)}=\mathcal{L}_{kj}^{(\pm)}$. 
Set 
$\mathcal{L}_{jk}:=\mathcal{L}_{jk}^{(+)}\cup\mathcal{L}_{jk}^{(-)}$ 
for short. 
\end{itemize}
\end{lemma}
\begin{proof}
Consider the motion of the hyperplane keeping the tangential contact 
with both $\Sigma_j$ and $\Sigma_j$. 
It would be intuitively obvious that 
$\mathcal{M}_{jk}^{(\pm)}$ and $B_{jk}^{(\pm)}$ 
are connected and disjoint respectively. 
We mainly deal with the $+$ case. 
So we shall show that $\mathcal{M}_{jk}^{(+)}$ 
is an $(n-2)$-dimensional submanifold of $\Sigma_j\times\Sigma_k$, 
and give its local coordinates. 
\par
Let $x=(x_1,\dotsc,x_n) \in\mathbb{R}^n$, 
let $a$ be a positive constant, 
and set $e_1=(1,0,\dotsc,0)\in\mathbb{R}^{n-1}$. 
Suppose that 
$$
\bigl((0,0,\dotsc,0,0),(a,0,\dotsc,0,0)\bigr)\in\mathcal{M}_{jk}^{(+)}, 
\quad
\nu_j(0,0,\dotsc,0,0)=(0,0,\dotsc,0,1).
$$
Without loss of generality, we may assume that there exist smooth functions 
$f(x_1,\dotsc,x_{n-1})$ and $g(x_1,\dotsc,x_{n-1})$ such that  
\begin{align*}
  \Sigma_j
& =
  \bigl\{x_n=f(x_1,\dotsc,x_{n-1})\bigr\}
  \quad\text{near}\quad 0\in\mathbb{R}^n, 
\\
  \Sigma_k
& =
  \bigl\{x_n=g(x_1,\dotsc,x_{n-1})\bigr\}
  \quad\text{near}\quad (ae_1,0)\in\mathbb{R}^n, 
\end{align*}
and 
$$
f(0)=g(ae_1)=0, 
\quad
f^\prime(0)=g^\prime(ae_1)=0, 
\quad
f^{\prime\prime}(0)<0, g^{\prime\prime}(ae_1)<0. 
$$
where $f^\prime$ the gradient vector field of $f$, 
and $f^{\prime\prime}$ is the Hessian matrix of $f$. 
Using local coordinates 
$u=(u_1,\dotsc,u_{n-1}), v=(v_1,\dotsc,v_{n-1}) \in \mathbb{R}^{n-1}$, 
we have 
$$
\Sigma_j\times\Sigma_k
=
\Bigl\{\Bigl(\bigl(u,f(u)\bigr),\bigl(v,g(v)\bigr)\Bigr)\Bigr\} 
\quad\text{near}\quad
\bigl(0,(ae_1,0)\bigr). 
$$ 
Using these coordinates, we deduce that 
$$
\nu_j\bigl(u,f(u)\bigr)=\bigl(-f^\prime(u),1\bigr), 
$$
and that the equation of $\Pi_j\bigl(u,f(u)\bigr)$ for $\bigl(v,g(v)\bigr)$ is given by 
$$
g(v)-f(u)-f^\prime(u)(v-u)=0. 
$$
Hence 
$\Bigl(\bigl(u,f(u)\bigr),\bigl(v,g(v)\bigr)\Bigr)\in\mathcal{M}_{jk}^{(+)}$ 
is characterized by 
$$
F(u,v)
:=
\begin{bmatrix}
g^\prime(v)-f^\prime(u)
\\
g(v)-f(u)-f^\prime(u)\cdot(v-u) 
\end{bmatrix}
=
0 
\quad\text{near}\quad
(u,v)=(0,ae_1). 
$$
Here $F(u,v)$ is thought to be an $n$-dimensional column vector valued function. 
We compute the rank of the gradient matrix $F^\prime(0,ae_1)$. 
We deduce that 
\begin{align*}
  F^\prime(u,v)
& =
  \begin{bmatrix}
  -f^{\prime\prime}(u) 
  & 
  g^{\prime\prime}(v)
  \\
  -f^\prime(u)^T-(v-u)f^{\prime\prime}(u)+f^\prime(u)^T 
  & 
  g^\prime(v)^T-f^\prime(u)^T
  \end{bmatrix} 
\\
& =
  \begin{bmatrix}
  -f^{\prime\prime}(u) 
  & 
  g^{\prime\prime}(v)
  \\
  -(v-u)f^{\prime\prime}(u) 
  & 
  0
  \end{bmatrix}. 
\end{align*}
Here we used $g^\prime(v)-f^\prime(u)=0$ of $F(u,v)=0$. 
Set $w_0:=-a\partial{f^\prime}/\partial u_1(0)$ for short. 
This is the first column of the matrix $f^{\prime\prime}(0)$ multiplied by $-a$. 
Then $w_0\ne0$ since $a>0$ and $f^{\prime\prime}(0)<0$. 
Then we have 
$$
F^\prime(0,ae_1)
=
\begin{bmatrix}
-f^{\prime\prime}(0) 
& 
g^{\prime\prime}(ae_1)
\\
w_0^T
& 
0
\end{bmatrix}. 
$$
Let $E_{n-1}$ and $O_{n-1}$ be 
the $(n-1)\times(n-1)$ identity matrix 
and the $(n-1)\times(n-1)$ zero matrix respectively. 
Note that $\det\bigl(g^{\prime\prime}(ae_1)\bigr)\ne0$ 
since $g^{\prime\prime}(ae_1)<0$. 
Multiply the above by a regular $(2n-2){\times}(2n-2)$ matrix from the right, we have   
$$
F^\prime(0,ae_1)
\begin{bmatrix}
E_{n-1} & O_{n-1}
\\
g^{\prime\prime}a(ae_1)^{-1}f^{\prime\prime}(0) & E_{n-1} 
\end{bmatrix}
=
\begin{bmatrix}
O_{n-1} & g^{\prime\prime}(ae_1)
\\
w_0^T & 0 
\end{bmatrix}, 
$$
which shows that $\operatorname{rank}\bigl(F^\prime(0,ae_1)\bigr)=n$. 
The implicit function theorem implies 
that $\mathcal{M}_{jk}^{(+)}$ is an $(n-2)$-dimensional submanifold of  
$\Sigma_j\times\Sigma_k$ near $(u,v)=(0,ae_1)$. 
This argument can be applied to the neighborhood of any point of 
$\mathcal{M}_{jk}^{(+)}$. 
Hence $\mathcal{M}_{jk}^{(+)}$ is an $(n-2)$-dimensional submanifold of 
$\Sigma_j\times\Sigma_k$ globally.  
\par
Next we consider local coordinates of $\mathcal{M}_{jk}^{(+)}$ 
for sufficiently small $u$ and $v-ae_1$.  
Fix arbitrary small $u^\prime:=(u_2,\dotsc,u_{n-1})\ne0$. 
The above arguments imply that there exists a pair of 
$u_1\in\mathbb{R}$ near $0$ and 
$v\in\mathbb{R}^{n-1}$ near $ae_1$ such that 
$\Bigl(\bigl((u_1,u^\prime),f(u_1,u^\prime)\bigr),\bigl(v,g(v)\bigr)\Bigr) \in \mathcal{M}_{jk}^{(+)}$. 
Set $s(u^\prime):=u_1$ and $z(u^\prime):=v$. 
If the pair of $u_1$ and $v$ is unique, then we have 
$$
\mathcal{M}_{jk}^{(+)}
=
\Bigl\{
\Bigl(
\bigl(s(u^\prime),u^\prime,f(s(u^\prime),u^\prime)\bigr),
\bigl(z(u^\prime),g(z(u^\prime))\bigr)
\Bigr)
\Bigr\}
\quad\text{near}\quad
\bigl(0,(ae_1,0)\bigr)
\in 
\mathbb{R}^n\times\mathbb{R}^n,
$$
$$
B_{jk}^{(+)}
=
\bigl\{
\bigl(s(u^\prime),u^\prime,f(s(u^\prime),u^\prime)\bigr)
\bigr\}
\quad\text{near}\quad
0
\in 
\mathbb{R}^n. 
$$
This shows that $u^\prime \in \mathbb{R}^{n-1}$ plays a role of local coordinates 
of both $\mathcal{M}_{jk}^{(+)}$ and $B_{jk}^{(+)}$. 
It follows that the projection of ${M}_{jk}^{(+)}$ onto $B_{jk}^{(+)}$ 
is diffeomorphic and $B_{jk}^{(+)}$ is an $(n-2)$-dimensional submanifold of $\Sigma_j$. 
\par
So we shall show the uniqueness of the pair of $u_1$ and $v$. 
It suffices to prove the uniqueness of $u_1$ since $v$ 
is uniquely determined by $(u_1,u^\prime)$. 
We express the equation of $\Pi_j\bigl(u,f(u)\bigr)$ of the form 
$x_n=G(u,v)$, $v=(v_1,v^\prime)\in\mathbb{R}^{n-1}$ near $v=ae_1$: 
$$
G(u,v)
=
f(u)+f^\prime(u)\cdot(v-u). 
$$ 
We consider a small perturbation of $u_1$ of the form $u_1+t$. 
Then we have 
\begin{align*}
  G(u+te_1,v)-G(u,v)
& =
  f(u+te_1)+f^\prime(u+te_1)\cdot(v_1-u_1-t,v^\prime-u^\prime)
\\
& -
  f(u)-f^\prime(u)\cdot(v_1-u_1,v^\prime-u^\prime)
\\
& =
  \{f(u+te_1)-f(u)-tf_{u_1}(u+te_1)\}
\\
& +
  \{f^\prime(u+te_1)-f^\prime(u)\}\cdot(v_1-u_1,v^\prime-u^\prime)
\\
& =
  tf^{\prime}_{u_1}(u)\cdot(v_1-u_1,v^\prime-u^\prime)
  +
  \mathcal{O}(t^2)
\\
& =
  tf^{\prime}_{u_1}(u)\cdot(a+v_1-a-u_1,v^\prime-u^\prime)
  +
  \mathcal{O}(t^2)
\\
& =
  t\{af_{u_1u_1}(u)+f^{\prime}_{u_1}(u)\cdot(v-ae_1-u)\}
  +
  \mathcal{O}(t^2). 
\end{align*}
It follows that $f_{u_1u_1}(0)<0$ since $f(u_1,0)$ is strictly concave in $u_1$. 
So $af_{u_1u_1}(u)<0$ for sufficiently small $u$. 
Hence $G(u+te_1,v)$ is strictly decreasing in $t$ near $t=0$ provided that 
$\lvert{v-ae_1}\rvert$ and $\lvert{u}\rvert$ are sufficiently small. 
This shows that for any fixed $u^\prime$, 
$(u_1+t,u^\prime)$ does not belong to $B_{jk}^{(+)}$ for $t\ne0$. 
The uniqueness of $u_1$ has been proved.  
\par
For $\mathcal{M}_{jk}^{(-)}$ we replace $g$ by $-g$ in the above setting 
and discuss in the same way. 
The claim for $\mathcal{L}_{jk}^{(\pm)}$ would be obvious. 
We omit the detail. 
\end{proof}
We denote by $\Omega_{jk}$ the subdomain of 
$\Sigma_j$ with boundary $B_{jk}^{(+)}{\cup}B_{jk}^{(-)}$. 
For any $y \in\Omega_{jk}$, 
the intersection of $\Pi_j(y)$ and $D_k$ is a section of $D_k$, 
and there exist different points $y_k^{(1)}, y_k^{(2)} \in \Pi_j(y)\cap\Sigma_k$ 
such that two lines connecting $y$ and $y_k(l)$, $l=1,2$ are tangential to $\Sigma_k$. 
Moreover we have $\Pi_j(y)\cap\overline{D_k}=\emptyset$ 
for any $y \in \Sigma_j\setminus\overline{\Omega_{jk}}$. 
\par
Note that $\mathcal{L}_{jk}$ is a union of a cone surface and a cylinder surface, 
or a union of two cone surfaces. 
These hypersurfaces are tangent to both $\Sigma_j$ and $\Sigma_k$. 
For example, on one hand, 
if $\Sigma_j$ and $\Sigma_k$ are spheres of the same size, 
then $\mathcal{L}_{jk}^{(+)}$ becomes a cylinder surface and 
$\mathcal{L}_{jk}^{(-)}$ becomes a cone surface. 
On the other hand, if $\Sigma_j$ and $\Sigma_k$ are spheres 
but their sizes are different, 
then both $\mathcal{L}_{jk}^{(\pm)}$ become cone surfaces. 
It is easy to see that $\mathcal{L}_{jk}=\mathcal{L}_{kj}$.  
\par
We now show the figure illustrating the above. 
\begin{center}
\includegraphics[width=100mm]{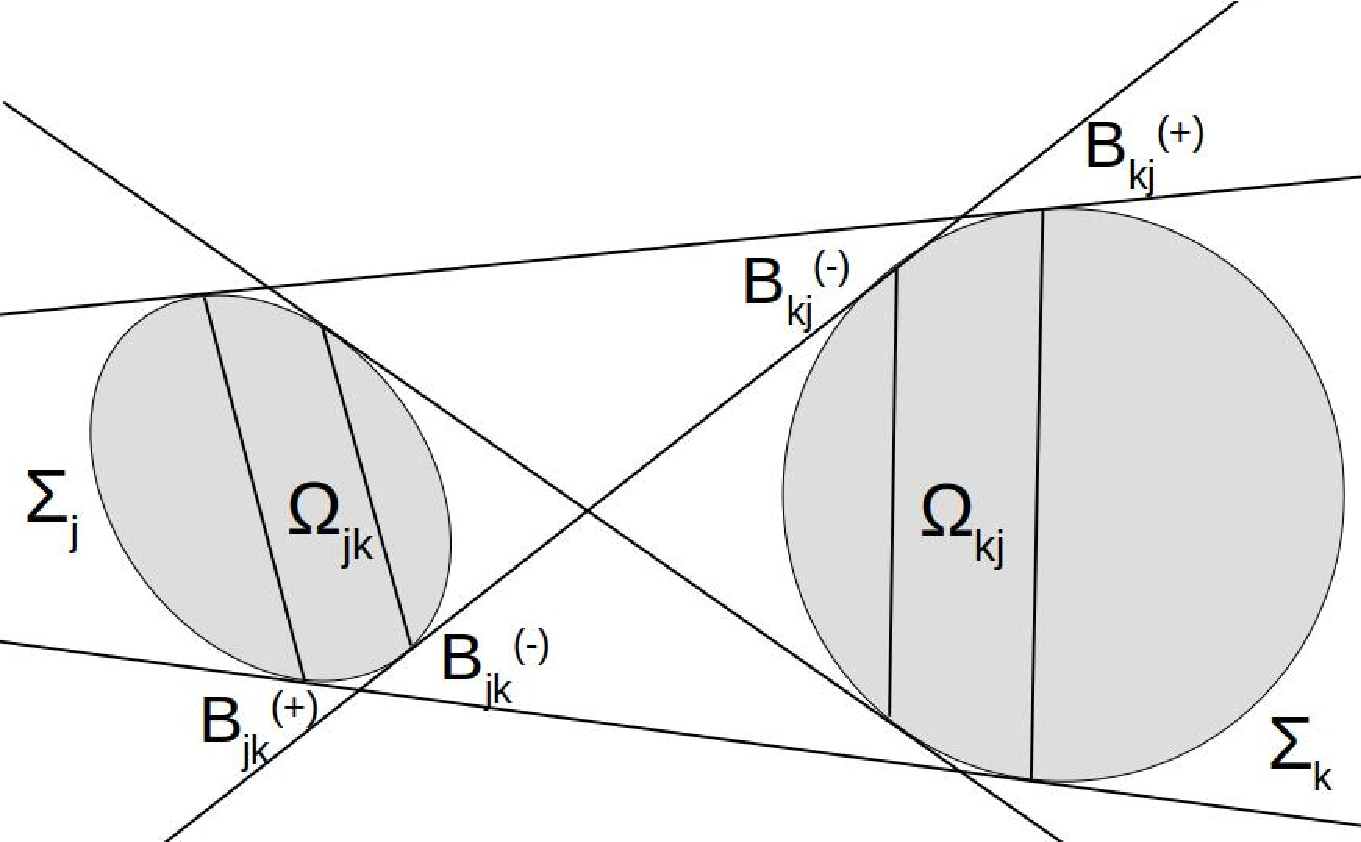}
\\
Figure~4. 
This illustrates the relationship between 
$\Sigma_j$, 
$\Sigma_k$, 
$B_{jk}^{(\pm)}$, 
$B_{kj}^{(\pm)}$, 
$\Omega_{jk}$ 
and 
$\Omega_{kj}$.
\end{center}
\vspace{11pt}
\par
We now show the basic results on the composition. 
\begin{lemma}
\label{theorem:canonical2}
\quad 
Compositions 
$(\Lambda_\phi^\prime)^\ast{\circ}N^\ast{S_j}$, 
$(\Lambda_\phi^\prime)^\ast{\circ}N^\ast{S_{jk}^{(1)}}$ 
and 
$(\Lambda_\phi^\prime)^\ast{\circ}N^\ast{S_{jk}^{(2)}}$ 
are clean with excesses $d(n-1-d)$, $d(n-1-d)$ and $d(n-2-d)$ respectively.  
They are conic Lagrangian submanifolds of $T^\ast\mathbb{R}^n\setminus0$ given by 
\begin{align*}
  (\Lambda_\phi^\prime)^\ast{\circ}N^\ast{S_j}
& =
  N^\ast\Sigma_j\setminus0, 
\\
  (\Lambda_\phi^\prime)^\ast{\circ}N^\ast{S_{jk}^{(1)}}
& =
  N^\ast\mathcal{L}_{jk}\setminus0, 
\\
  (\Lambda_\phi^\prime)^\ast{\circ}N^\ast{S_{jk}^{(2)}}
& =
  (N^\ast\Omega_{jk}\setminus0)\cup(N^\ast\Omega_{kj}\setminus0). 
\end{align*}
\end{lemma}
\begin{proof}
We express $(\Lambda_\phi^\prime)^\ast$ 
in $T^\ast\mathbb{R}^n{\times}T^\ast{G(d,n)}$ 
as 
\begin{align*}
  (\Lambda_\phi^\prime)^\ast
& =
  \Bigl\{
  \Bigl(
  (x^{\prime\prime}+t_1\omega_1+\dotsb+t_d\omega_d,\xi), 
  \bigl(\sigma,x^{\prime\prime},\xi(t_1,\dotsc,t_d,1)\bigr)
  \Bigr) :
\\
& \qquad
  (\sigma,x^{\prime\prime}) \in G(d,n), 
  \sigma=\langle\omega_1,\dotsc,\omega_d\rangle_\text{ON} \in G_{d,n},
\\
& \qquad  
  t_1,\dotsc,t_d\in\mathbb{R},\  
  \xi\in\sigma^\perp\setminus\{0\}
  \Bigr\}. 
\end{align*}
Let $t_0$, $t_j$ and $t_k$ be the same as in \eqref{equation:t0}. 
\par
We consider $(\Lambda_\phi^\prime)^\ast{\circ}N^\ast{S_j}$. 
Since 
\begin{align*}
  N^\ast{S_j}\setminus0
& =
  \bigl\{
  \bigl(
  (\sigma,y_j-\pi_\sigma{y_j},\eta_j(t_0\omega_0+t_j\omega_1,1)
  \bigr) : 
\\
& \qquad
  (y_j,\eta_j) \in N^\ast\Sigma_j\setminus0, 
  \sigma \in G_{d,n}\cap\eta_j^\perp
  \bigr\}, 
\end{align*}
we have 
\begin{align*}
  M_2
& := 
  \Bigl(T^\ast\mathbb{R}^n\times\Delta\bigl(T^\ast{G(d,n)}\bigr)\Bigr)
  \cap
  \bigl((\Lambda_\phi^\prime)^\ast{\times}N^\ast{S_j}\bigr) 
\\
& =
  \Bigl\{
  \Bigl(
  (y_j,\eta_j),
  \bigl(\sigma,y_j-\pi_\sigma{y_j},\eta_j(t_0\omega_0+t_j\omega_1,1)\bigr),
\\
& \qquad 
  \bigl(\sigma,y_j-\pi_\sigma{y_j},\eta_j(t_0\omega_0+t_j\omega_1,1)\bigr)
  \Bigr) : 
\\
& \qquad
  (y_j,\eta_j) \in N^\ast\Sigma_j\setminus0, 
  \sigma \in G_{d,n}\cap\eta_j^\perp  
  \Bigr\},  
\end{align*}
$$
(\Lambda_\phi^\prime)^\ast{\circ}N^\ast{S_j}
=
\{
(y,\eta): (y,\eta) \in N^\ast\Sigma_j\setminus0
\}
=
N^\ast\Sigma_j\setminus0,
$$ 
$$
\dim(M_2)
=
\dim(N^\ast\Sigma_j\setminus0)+\dim(G_{d,n-1})
=
n+d(n-1-d). 
$$
Then $(\Lambda_\phi^\prime)^\ast{\circ}N^\ast{S_j}$ is clean with excess $d(n-1-d)$ 
and a conic Lagrangian submanifold of $T^\ast\mathbb{R}^n\setminus0$. 
\par
We consider $(\Lambda_\phi^\prime)^\ast{\circ}N^\ast{S_{jk}^{(1)}}$. 
Lemma~\ref{theorem:conormal2} implies that 
\begin{align*}
  N^\ast{S_{jk}^{(1)}}\setminus0
& =
  \bigl\{
  \bigl(
  \sigma,y_j-\pi_\sigma{y_j}, 
  \eta_j(t_0\omega_0+t\omega_1)
  \bigr) : 
\\
& \qquad
  y_j \in B_{jk}, \eta_j \in N^\ast_{y_j}\Sigma_j\setminus0, 
  \sigma \in G_{d,n}\cap\eta_j^\perp, t\in\mathbb{R}
  \bigr\}. 
\end{align*}
Then we have 
\begin{align*}
  M_3
& := 
  \Bigl(T^\ast\mathbb{R}^n\times\Delta\bigl(T^\ast{G(d,n)}\bigr)\Bigr)
  \cap
  \bigl((\Lambda_\phi^\prime)^\ast{\times}N^\ast{S_{jk}^{(1)}}\bigr) 
\\
& =
  \Bigl\{
  \Bigl(
  (y_j+t\omega_1,\eta_j),
  \bigl(\sigma,y_j-\pi_\sigma{y_j},\eta_j(t_0\omega_0+t\omega_1,1)\bigr),
\\
& \qquad
  \bigl(\sigma,y_j-\pi_\sigma{y_j},\eta_j(t_0\omega_0+t\omega_1,1)\bigr)
  \Bigr) : 
\\
& \qquad
  y_j \in B_{jk}, \eta_j \in N_{y_j}^\ast\Sigma_j\setminus0, 
  \sigma \in G_{d,n}\cap\eta_j^\perp, t\in\mathbb{R}  
  \Bigr\},  
\end{align*}
\begin{align*}
  (\Lambda_\phi^\prime)^\ast{\times}N^\ast{S_{jk}^{(1)}}
& =
  \{
  (y_j+t(y_k-y_j),\eta_j): 
  (y_j,y_k) \in \mathcal{M}_{jk},\ t\in\mathbb{R}, 
  \eta_j \in N^\ast_{y_j}\Sigma_j\setminus0\}
\\
& =
  N^\ast\mathcal{L}_{jk}\setminus0. 
\end{align*}
Hence we obtain 
$$
\dim(M_3)
=
\dim(B_{jk})+2+\dim(G_{d,n-1})=n+d(n-1-d). 
$$
Thus $(\Lambda_\phi^\prime)^\ast{\circ}N^\ast{S_{jk}^{(1)}}$ 
is clean with excess $d(n-1-d)$ and 
a conic Lagrangian submanifold of $T^\ast\mathbb{R}^n\setminus0$. 
\par
We consider $(\Lambda_\phi^\prime)^\ast{\circ}N^\ast{S_{jk}^{(2)}}$. 
In this case $n \geqq d+2$ is assumed. 
Lemma~\ref{theorem:conormal2} implies that 
\begin{align*}
  N^\ast{S_{jk}^{(2)}}\setminus0
& =
  \bigl\{
  \bigl(
  \sigma,y_j-\pi_\sigma{y_j}, 
  \Xi
  \bigr) : 
  \Xi \in \langle\eta_j(\pi_\sigma{y_j},1),\eta_k(\pi_\sigma{y_k},1)\rangle,
\\
& \qquad
  y_j \in \Omega_{jk} \ \text{with some}\ y_k\in\Sigma_k \ \text{or}\ 
  y_k \in \Omega_{kj} \ \text{with some}\ y_j\in\Sigma_j, 
\\ 
& \qquad
  (y_j,\eta_j) \in N^\ast\Sigma_j\setminus0, 
  (y_k,\eta_k) \in N^\ast\Sigma_k\setminus0, 
  \sigma \in G_{d,n}\cap\eta_j^\perp\cap\eta_k^\perp
  \bigr\}. 
\end{align*}
Then we have 
\begin{align*}
  M_4
& := 
  \Bigl(T^\ast\mathbb{R}^n\times\Delta\bigl(T^\ast{G(d,n)}\bigr)\Bigr)
  \cap
  \bigl((\Lambda_\phi^\prime)^\ast{\times}N^\ast{S_{jk}^{(2)}}\bigr) 
\\
& =
  \Bigl\{
  \Bigl(
  (y,\eta),
  \bigl(\sigma,y_j-\pi_\sigma{y_j},\eta(\pi_\sigma{y},1)\bigr),
  \bigl(\sigma,y_j-\pi_\sigma{y_j},\eta(\pi_\sigma{y},1)\bigr)
  \Bigr) : 
\\
& \qquad
  (y,\eta)=(y_j,\eta_j) \in N^\ast\Sigma_j\setminus0,\ 
  y_j \in \Omega_{jk} \ \text{with some}\ y_k\in\Sigma_k \ \text{or}\ 
\\
& \qquad
  (y,\eta)=(y_k,\eta_k) \in N^\ast\Sigma_k\setminus0,\ 
  y_k \in \Omega_{kj} \ \text{with some}\ y_j\in\Sigma_, 
\\
& \qquad
  \sigma \in G_{d,n}\cap\eta_j^\perp\cap\eta_k^\perp
  \Bigr\},  
\end{align*}
\begin{align*}
  (\Lambda_\phi^\prime)^\ast{\circ}N^\ast{S_{jk}^{(2)}}
& =
  \{(y_j,\eta_j) \in N^\ast\Sigma_j\setminus0 : 
  y_j \in \Omega_{jk} \ \text{with some}\ y_k\in\Sigma_k \}
\\
& \ \cup
  \{(y_k,\eta_k) \in N^\ast\Sigma_k\setminus0 :  
  y_k \in \Omega_{kj} \ \text{with some}\ y_j\in\Sigma_j\}
\\
& =
  (N^\ast\Omega_{jk}\setminus0)\cup (N^\ast\Omega_{kj}\setminus0). 
\end{align*}
Hence we obtain 
$$
\dim(M_4)
=
\dim(N^\ast\Omega_{jk}\setminus0)+\dim(G_{d,n-2})=n+d(n-2-d).
$$ 
Thus $(\Lambda_\phi^\prime)^\ast{\circ}N^\ast{S_{jk}^{(2)}}$ 
is clean with excess $d(n-2-d)$ and 
a conic Lagrangian submanifold of $T^\ast\mathbb{R}^n\setminus0$. 
\end{proof} 
%
%
\section{Paired Lagrangian distributions and their products}
\label{section:lagrangian}
In the present section we study basic facts used for our quantitative analysis of the metal streaking artifacts. We begin with identifying the classes of conormal distributions to which $\chi_D$ and $\mathcal{R}_d\chi_D$ belong respectively. 
\begin{lemma}
\label{theorem:chid}
\begin{align*}
  \chi_{D_j} 
& \in 
  I^{-1/2-n/4}(N^\ast\Sigma_j), 
\\
  \mathcal{R}_d\chi_{D_j} 
& \in 
  I^{-(d+2)/2+1/2-N(d,n)/4}(N^\ast{S}_j)
  =
  I^{-(d+1)(n-d+2)/4}(N^\ast{S_j}). 
\end{align*}
Moreover the principal symbols of these conormal distributions do not vanish everywhere on the associating conormal bundles respectively. 
\end{lemma} 
\begin{proof}[{\bf Proof}]
It is easy to see that 
$$
\chi_{D_j} 
\in 
I^{-1+1/2-n/4}(N^\ast\Sigma_j)
=
I^{-1/2-n/4}(N^\ast\Sigma_j). 
$$
Indeed if we choose appropriate local coordinates in $x\in\mathbb{R}^n$, then 
$\chi_{D_j}$ can be expressed as a characteristic function of a half space 
$\{(0,x_2,\dotsc,x_n) : x_2,\dotsc,x_n\in\mathbb{R}\}$ and 
$$
\chi_{D_j}(x)
=
\int_{\mathbb{R}}
e^{ix_1\xi_1}
a(x_2,\dotsc,x_n,\xi_1)
d\xi_1
$$
locally near $(x_1,x_2,\dotsc,x_n)=(0,0,\dotsc,0)$ 
with some amplitude 
$$
a(x_2,\dotsc,x_n,\xi_1)
\sim
\frac{1}{\xi_1}
\in 
S^{-1}\bigl(\mathbb{R}^{n-1}\times(\mathbb{R}\setminus\{0\})\bigr). 
$$
It is also easy to see that the principal symbol of $a(x_2,\dotsc,x_n,\xi_1)$ does not vanish. 
Combining this, Lemma~\ref{theorem:hoermander2523}, 
Theorem~\ref{theorem:canonicalrd} and Lemma~\ref{theorem:lemma1}, 
we deduce that 
$$
\mathcal{R}_d\chi_{D_j}
\in 
I^{-(N(d,n)+n)/4+(n-d)/2-1/2-n/4}(\Lambda_\phi^\prime{\circ}N^\ast\Sigma_j) 
=
I^{-(d+2)/2+1/2-N(d,n)/4}(N^\ast{S_j}), 
$$
and the principal symbol of $\mathcal{R}_d\chi_{D_j}$ does not vanish since the principal symbol of $\chi_{D_j}$ does not vanish and $\mathcal{R}_d$ is an elliptic Fourier integral operator. This completes the proof. 
\end{proof}
In the present section we concentrate on the analysis of 
$$
(\mathcal{R}_d\chi_D)^2
=
\sum_{j=1}^J
(\mathcal{R}_d\chi_{D_j})^2
+
2
\sum_{1 \leqq j < k \leqq J}
\mathcal{R}_d\chi_{D_j} \cdot \mathcal{R}_d\chi_{D_k}. 
$$
So we need to study 
$(\mathcal{R}_d\chi_{D_j})^2$ and $\mathcal{R}_d\chi_{D_j} \cdot \mathcal{R}_d\chi_{D_k}$ with $j{\ne}k$. We begin with the analysis of $(\mathcal{R}_d\chi_{D_j})^2$. This will be done by symbolic calculus of conormal distributions. In general we have the following. 
\begin{lemma}
\label{theorem:lemma32}
If $u, v \in I^{-(d+1)(n-d+2)/4}(N^\ast{S_j})$, then $uv \in I^{-(d+1)(n-d+2)/4}(N^\ast{S_j})$.
\end{lemma}
\begin{proof}[{\bf Proof}]
The outline of the proof of Lemma~\ref{theorem:lemma32} is similar to that of \cite[Lemma~4.1]{PalaciosUhlmannWang}. However we introduce some cut-off functions for the frequency space and obtain fine evaluation. Suppose that $u, v \in I^{-(d+1)(n-d+2)/4}(N^\ast{S_j})$. 
Fix arbitrary $x_0 \in S_j$ and consider the product $uv$ near $x_0$. Note that the codimension of $S_j$ is one and the order of the amplitudes of $u$ and $v$ is 
$$
-
\frac{(d+1)(n-d+2)}{4}
-
\frac{1}{2}
+
\frac{N(d,n)}{4}
=
-\frac{d}{2}-1. 
$$
Then there exist amplitudes 
$a(x,\xi_1), b(x,\xi_1) \in S^{-d/2-1}(\mathbb{R}^{N(d,n)}\times\mathbb{R})$ which are compactly supported in $x$ near $x_0$ such that $u(x)$ and $v(x)$ are given by 
$$
u(x)
=
\int_{\mathbb{R}}
e^{ix_1\xi_1}
a(x,\xi_1)
d\xi_1,
\quad
v(x)
=
\int_{\mathbb{R}}
e^{ix_1\xi_1}
b(x,\xi_1)
d\xi_1 
$$
near $x=x_0$ respectively. 
Hence we have the explicit formula of the product $uv$ as 
$$
u(x)v(x)
=
\int_{\mathbb{R}}
\int_{\mathbb{R}}
e^{ix_1(\xi_1+\eta_1)}
a(x,\xi_1)b(x,\eta_1)
d\xi_1
d\eta_1
=
\int_{\mathbb{R}}
e^{ix_1\xi_1}
c_1(x,\xi_1)
d\xi_1,
$$
$$
c_1(x,\xi_1)
=
\int_{\mathbb{R}}
a(x,\xi_1-\eta_1)
b(x,\eta_1)
d\eta_1
$$
near $x=x_0$. 
It is easy to see that $c_1(x,\xi_1)$ is compactly supported in $x$ near $x_0$. It suffices to show that $c_1(x,\xi_1) \in S^{-d/2-1}(\mathbb{R}^{N(d,n)}\times\mathbb{R})$. For this purpose we split $\eta_1$-space into two parts by using cut-off functions. Pick up $\psi(t) \in C^\infty_0(\mathbb{R})$ such that 
$$
0\leqq\psi(t)\leqq1, 
\quad
\psi(t)
\equiv
\begin{cases}
1 &\ (\lvert{t}\rvert\leqq1/2),
\\
0 &\ (\lvert{t}\rvert\geqq3/4),
\end{cases}
\quad
\psi(t)>0 
\quad
(\lvert{t}\rvert<3/4).
$$
Set $\Psi_1(\xi_1,\eta_1):=\psi(\langle\xi_1-\eta_1\rangle/\langle\xi_1\rangle)$. Then the properties of $\Psi_1(\xi_1,\eta_1)$ are the following. 
\begin{itemize}
\item[(i)] 
$\Psi_1(\xi_1,\eta_1)>0$ is equivalent to $\langle\xi_1-\eta_1\rangle\leqq3\langle\xi_1\rangle/4$, which implies $C_1^{-1}\langle\xi_1\rangle \leqq \langle\eta_1\rangle \leqq C_1\langle\xi_1\rangle$ with some $C_1>1$. Indeed $\langle\xi_1-\eta_1\rangle\leqq3\langle\xi_1\rangle/4$ implies that $\lvert\xi_1-\eta_1\rvert\leqq3\lvert\xi_1\rvert/4$ since 
$$
1+\lvert\xi_1-\eta_1\rvert^2
\leqq
\frac{9}{16}+\frac{9\lvert\xi_1\rvert^2}{16}
<
1+\frac{9\lvert\xi_1\rvert^2}{16}. 
$$
Applying this to $\pm(\lvert\xi_1\rvert-\lvert\eta_1\rvert)\leqq\lvert\xi_1-\eta_1\rvert$, we have $\lvert\xi_1\rvert/4 \leqq \lvert\eta_1\rvert \leqq 7\lvert\xi_1\rvert/4$. 
\item[(ii)] 
$1-\Psi_1(\xi_1,\eta_1)>0$ is equivalent to $\langle\xi_1-\eta_1\rangle>\langle\xi_1\rangle/2$.
\item[(iii)] 
The results of (i) and (ii) implies that 
$\operatorname{supp}\psi^\prime(\langle\xi_1-\eta_1\rangle/\langle\xi_1\rangle) \subset E_1$, 
where 
$$
E_1
=
\{
(\xi_1,\eta_1)\in\mathbb{R}^2 
: 
\langle\xi_1\rangle/2\leqq\langle\xi_1-\eta_1\rangle\leqq3\langle\xi_1\rangle/4
\}.
$$
In particular for $(\xi_1,\eta_1) \in E_1$
\begin{equation}
C_1^{-1}\langle\eta_1\rangle/2 \leqq \langle\xi_1\rangle/2 \leqq \langle\xi_1-\eta_1\rangle \leqq 3\langle\xi\rangle/4 \leqq 3C_1\langle\eta_1\rangle/4.
\label{equation:equivalence1}
\end{equation}
\item[(iv)] 
Let $\alpha$ and $\beta$ be non-negative integers with $\alpha+\beta>0$. 
By using the chain rule of differentiation and \eqref{equation:equivalence1}, we have 
\begin{align*}
  \partial_{\xi_1}^\alpha\partial_{\eta_1}^\beta\Psi_1(\xi_1,\eta_1)
& =
  \sum_{l=1}^{\alpha+\beta}
  \sum_{\substack{\alpha_1+\dotsb+\alpha_l=\alpha \\ \beta_1+\dotsb+\beta_l=\beta}}
  C^{\alpha_1,\dotsc,\alpha_l}_{\beta_1,\dotsc,\beta_l}
  \cdot 
  \psi^{(l)}\left(\frac{\langle\xi_1-\eta_1\rangle}{\langle\xi_1\rangle}\right)
\\
& \times
  \prod_{m=1}^l
  \partial_{\xi_1}^{\alpha_m}
  \partial_{\eta_1}^{\beta_m}
  \left(\frac{\langle\xi_1-\eta_1\rangle}{\langle\xi_1\rangle}\right)
\\
& =
  \mathcal{O}(\langle\xi_1\rangle^{-(\alpha+\beta)}),
\end{align*}
where $C^{\alpha_1,\dotsc,\alpha_l}_{\beta_1,\dotsc,\beta_l}$, 
with $\alpha_1+\dotsb+\alpha_l=\alpha$ and $\beta_1+\dotsb+\beta_l=\beta$ 
are some constants. Recall that $0\leqq\Psi_1(\xi_1,\eta_1)\leqq1$. 
Then we deduce that $\partial_{\xi_1}^\alpha\partial_{\eta_1}^\beta\Psi_1(\xi_1,\eta_1)=\mathcal{O}(\langle\xi\rangle^{-(\alpha+\beta)})$ for all non-negative integers $\alpha$ and $\beta$. 
\end{itemize}
\par
We make use of 
the Leibniz formula,  
$\partial_{\xi_1}^\alpha a(x,\xi_1-\eta_1)=(-1)^\alpha\partial_{\eta_1}^\alpha a(x,\xi_1-\eta_1)$, integration by parts and the properties of $\Psi_1(\xi_1,\eta_1)$ in the order, and confirm that $\partial_x^\gamma\partial_{\xi_1}^\alpha c_1(x,\xi_1)=\mathcal{O}(\langle\xi_1\rangle^{-d/2-1-\alpha})$. 
We deduce that  
\begin{align}
  \partial_x^\gamma\partial_{\xi_1}^\alpha c_1(x,\xi_1)
& =
  \sum_{\gamma_1+\gamma_2=\gamma}
  \frac{\gamma!}{\gamma_1!\gamma_2!}
  \int_{\mathbb{R}}
  \partial_x^{\gamma_1}\partial_{\xi_1}^\alpha a(x,\xi_1-\eta_1)
  \partial_x^{\gamma_2}b(x,\eta_1)
  d\eta_1
\nonumber
\\
& = 
  F_1+G_1,
\label{equation:split100}
\\
  F_1
& =
  \sum_{\gamma_1+\gamma_2=\gamma}
  \frac{\gamma!}{\gamma_1!\gamma_2!}
  \int_{\mathbb{R}}
  \bigl(1-\Psi_1(\xi_1,\eta_1)\bigr)
  \partial_x^{\gamma_1}\partial_{\xi_1}^\alpha a(x,\xi_1-\eta_1)
  \partial_x^{\gamma_2}b(x,\eta_1)
  d\eta_1,
\nonumber
\\
  G_1
& =
  \sum_{\gamma_1+\gamma_2=\gamma}
  \frac{\gamma!}{\gamma_1!\gamma_2!}
  \int_{\mathbb{R}}
  \Psi_1(\xi_1,\eta_1)
  \partial_x^{\gamma_1}(-\partial_{\eta_1})^\alpha a(x,\xi_1-\eta_1)
  \partial_x^{\gamma_2}b(x,\eta_1)
  d\eta_1.
\nonumber
\end{align}
The region of the integration $F_1$ satisfies $\langle\xi_1-\eta_1\rangle>\langle\xi_1\rangle/2$. 
We deduce that 
\begin{align}
  F_1
& =
  \int_{\langle\xi_1-\eta_1\rangle>\langle\xi_1\rangle/2}
  \mathcal{O}(\langle\xi_1-\eta_1\rangle^{-d/2-1-\alpha}\langle\eta_1\rangle^{-d/2-1})
  d\eta_1
\nonumber
\\
& =
  \int_{\langle\xi_1-\eta_1\rangle>\langle\xi_1\rangle/2}
  \mathcal{O}(\langle\xi_1\rangle^{-d/2-1-\alpha}\langle\eta_1\rangle^{-d/2-1})
  d\eta_1
\nonumber
\\
& =
  \mathcal{O}(\langle\xi_1\rangle^{-d/2-1-\alpha})
\label{equation:split101}
\end{align}
since $\langle\eta_1\rangle^{-d/2-1}$ is integrable in $\eta_1$ on $\mathbb{R}$. 
\par
The region of the integration $G_1$ satisfies 
$C_1^{-1}\langle\xi_1\rangle \leqq \langle\eta_1\rangle \leqq C_1\langle\xi_1\rangle$. 
We deduce that 
\begin{align}
  G_1
& =
  \sum_{\alpha_1+\alpha_2=\alpha}
  \frac{\alpha!}{\alpha_1!\alpha_2!}
  \sum_{\gamma_1+\gamma_2=\gamma}
  \frac{\gamma!}{\gamma_1!\gamma_2!}
\nonumber
\\
& \times
  \int_{\mathbb{R}}
  \partial_{\eta_1}^{\alpha_1}\Psi_1(\xi_1,\eta_1)
  \cdot
  \partial_x^{\gamma_1}a(x,\xi_1-\eta_1)
  \cdot
  \partial_x^{\gamma_2}\partial_{\eta_1}^{\alpha_2}b(x,\eta_1)
  d\eta_1
\nonumber
\\
& =
  \sum_{\alpha_1+\alpha_2=\alpha}
  \int_{C_1^{-1}\langle\xi_1\rangle \leqq \langle\eta_1\rangle \leqq C_1\langle\xi_1\rangle}
\nonumber
\\
& \qquad\qquad\qquad\times
  \mathcal{O}(\langle{\xi_1}\rangle^{-\alpha_1}\langle\xi_1-\eta_1\rangle^{-d/2-1}\langle\eta_1\rangle^{-d/2-1-\alpha_2})
  d\eta_1
\nonumber
\\
& =
  \int_{\mathbb{R}}
  \mathcal{O}(\langle\xi_1-\eta_1\rangle^{-d/2-1}\langle\xi_1\rangle^{-d/2-1-\alpha})
  d\eta_1
\nonumber
\\
& =
  \mathcal{O}(\langle\xi_1\rangle^{-d/2-1-\alpha})
\label{equation:split102}
\end{align}
since $\langle\xi_1-\eta_1\rangle^{-d/2-1}$ is integrable in $\eta_1$ on $\mathbb{R}$. 
\par
Substitute 
\eqref{equation:split101} 
and 
\eqref{equation:split102} 
into 
\eqref{equation:split100}. 
We obtain 
$\partial_x^\gamma\partial_{\xi_1}^\alpha c_1(x,\xi_1)=\mathcal{O}(\langle\xi_1\rangle^{-d/2-1-\alpha})$. 
This completes the proof. 
\end{proof} 
Next we study the microlocal singularity of $\mathcal{R}_d\chi_{D_j}\cdot\mathcal{R}_d\chi_{D_k}$ on $S_{jk}$. For this purpose we here introduce some classes of distributions. They are called paired Lagrangian distributions. See \cite{MelroseUhlmann}, \cite{GuilleminUhlmann} and \cite{GreenleafUhlmann}. 
\begin{definition}
\label{theorem:paired1} 
Let $X$ be a smooth manifold, and let $\mu, \nu \in \mathbb{R}$. 
Suppose that $\Lambda_0$ and $\Lambda_1$ are cleanly intersecting conic Lagrangian submanifold of $T^\ast{X}\setminus0$. We define a class of distributions $I^{\mu,\nu}(\Lambda_0,\Lambda_1)$ on $X$ which is said to be a paired Lagrangian distribution associated to the pair $(\Lambda_0,\Lambda)$ of order $\mu,\nu$ as follows. We say that $u \in I^{\mu,\nu}(\Lambda_0,\Lambda_1)$ if $u \in \mathcal{D}^\prime(X)$, $\operatorname{WF}(u) \subset \Lambda_0\cup\Lambda_1$ and microlocally away from $\Lambda_0\cap\Lambda_1$ 
$$
I^{\mu,\nu}(\Lambda_0,\Lambda_1)
\subset
I^{\mu+\nu}(\Lambda_0\setminus\Lambda_1),
\quad
I^{\mu,\nu}(\Lambda_0,\Lambda_1)
\subset
I^{\mu}(\Lambda_1). 
$$
\end{definition}
It is known that paired Lagrangian distributions associated to a pair of two conormal bundles can be characterized as oscillatory integrals as follows. Let $X$ be an $N$-dimensional smooth manifold, and let $Y$ and $Z$ be transversely intersecting submanifolds of $X$. 
It follows that $N^\ast{Y}$ and $N^\ast(Y{\cap}Z)$ intersect cleanly. 
Set $\operatorname{codim}{Y}=k$ and $\operatorname{codim}(Y{\cap}Z)=k+l$. 
Fix arbitrary $p_0 \in X$, and choose local coordinates $x=(x_1,\dotsc,x_N)=(x^\prime,x^{\prime\prime},x^{\prime\prime\prime})\in\mathbb{R}^k\times\mathbb{R}^l\times\mathbb{R}^{N-k-l}$ such that $x(p_0)=0$ and 
\begin{align*}
  Y
& =
  \{x=(x^\prime,x^{\prime\prime},x^{\prime\prime\prime}) : x^\prime=(x_1,\dotsc,x_k)=0\},
\\
  Y{\cap}Z
& =
  \{x=(x^\prime,x^{\prime\prime},x^{\prime\prime\prime}) : x^\prime=(x_1,\dotsc,x_k)=0, x^{\prime\prime}=(x_{k+1},\dotsc,x_{k+l})=0\}
\end{align*}
near $p_0$. It is known that $u \in I^{\mu,\nu}(N^\ast(Y{\cap}Z),N^\ast{Y})$ if and only if there exists a symbol $a(x,\xi^\prime,\xi^{\prime\prime}) \in S^{\mu-k/2+N/4,\nu-l/2}\bigl(\mathbb{R}^N\times(\mathbb{R}^k\setminus\{0\})\times\mathbb{R}^l\bigr)$ such that 
$$
u(x)
=
\iint_{\mathbb{R}^k\times\mathbb{R}^l}
e^{i(x^\prime\cdot\xi^\prime+x^{\prime\prime}\cdot\xi^{\prime\prime})}
a(x,\xi^\prime,\xi^{\prime\prime})
d\xi^\prime 
d\xi^{\prime\prime}
$$
near $x=0$. 
Here 
$S^{m,m^\prime}\bigl(\mathbb{R}^N\times(\mathbb{R}^k\setminus\{0\})\times\mathbb{R}^l\bigr)$ 
is the set of all smooth functions $a(x,\xi,\eta)$ on $\mathbb{R}^N\times(\mathbb{R}^k\setminus\{0\})\times\mathbb{R}^l$ with the following conditions: for any compact set $K$ in $\mathbb{R}^N$ and for any multi-indices $\alpha$, $\beta$, and $\gamma$ there exists a constant $C(K,\alpha,\beta,\gamma)>0$ such that 
$$
\lvert\partial_x^\gamma\partial_\xi^\alpha\partial_\eta^\beta a(x,\xi,\eta)\rvert
\leqq
C(K,\alpha,\beta,\gamma)
\langle\xi;\eta\rangle^{m-\lvert\alpha\rvert}
\langle\eta\rangle^{m^\prime-\lvert\beta\rvert}
$$
for $(x,\xi,\eta) \in K\times(\mathbb{R}^k\setminus\{0\})\times\mathbb{R}^l$. 
\par
We compute the microlocal singularity of 
$\mathcal{R}_d\chi_{D_j}\cdot\mathcal{R}_d\chi_{D_k}$ near $S_{jk}$. 
For this purpose we make full use of the following strong results due to Greenleaf and Uhlmann in \cite{GreenleafUhlmann}. 
\begin{lemma}[{\cite[Lemma~1.1]{GreenleafUhlmann}}]
\label{theorem:gu}
Let $X$ be an $N$-dimensional smooth manifold, and let $Y$ and $Z$ be transversely intersecting smooth submanifolds of $X$ with $\operatorname{codim}(Y)=k_1$ and $\operatorname{codim}(Z)=k_2$ respectively. Set $l_1=\operatorname{codim}(Y{\cap}Z)-k_1$ and $l_2=\operatorname{codim}(Y{\cap}Z)-k_2$. Let $m_1$ and $m_2$ be real numbers. Then we have 
\begin{align*}
&  I^{m_1+k_1/2-N/4}(N^\ast{Y}) \cdot I^{m_2+k_2/2-N/4}(N^\ast{Z})
\\
& \qquad \subset 
  I^{m_1+k_1/2-N/4,m_2+l_1/2}(N^\ast(Y{\cap}Z),N^\ast{Y})
\\
& \qquad +
  I^{m_2+k_2/2-N/4,m_1+l_2/2}(N^\ast(Y{\cap}Z),N^\ast{Z}).
\end{align*}
\end{lemma}
Applying Lemma~\ref{theorem:gu} to $\mathcal{R}_d\chi_{D_j}\cdot\mathcal{R}_d\chi_{D_k}$, we have the following results. 
\begin{theorem}
\label{theorem:quadratic}
\quad
\begin{itemize}
\item[{\rm (I)}] 
$\mathcal{R}_d\chi_{D_j}\cdot\mathcal{R}_d\chi_{D_k}$ belongs to 
\begin{align*}
  \mathcal{A}_{jk}
  :=
& I^{-(d+1)/2-N(d,n)/4,-(d+1)/2}(N^\ast{S_{jk}},N^\ast{S_j}) 
\\
  +
& I^{-(d+1)/2-N(d,n)/4,-(d+1)/2}(N^\ast{S_{jk}},N^\ast{S_k}).
\end{align*}
Moreover the principal symbol of the product is non-vanishing on $N^\ast{S_{jk}}$. 
\item[{\rm (II)}] 
$\mathcal{R}_d^\ast(-\Delta_{x^{\prime\prime}})^{d/2}\bigl\{\mathcal{R}_d\chi_{D_j}\cdot\mathcal{R}_d\chi_{D_k}\bigr\}$ belongs to 
\begin{align*}
  \mathcal{X}_{jk}
  :=
& I^{-(d+1)/2-n/4+d(n-d)/2,-(d+1)/2}(N^\ast\mathcal{L}_{jk}\setminus0,N^\ast\Sigma_j\setminus0) 
\\
  +
& I^{-(d+1)/2-n/4+d(n-d)/2,-(d+1)/2}(N^\ast\mathcal{L}_{jk}\setminus0,N^\ast\Sigma_k\setminus0).
\end{align*}
In particular,
$\mathcal{R}_d^\ast(-\Delta_{x^{\prime\prime}})^{d/2}\bigl\{\mathcal{R}_d\chi_{D_j}\cdot\mathcal{R}_d\chi_{D_k}\bigr\} \in I^{-(d+1)-n/4+d(n-d)/2}(N^\ast\mathcal{L}_{jk}\setminus0)$ microlocally away from $(N^\ast\Sigma_j\setminus0){\cup}(N^\ast\Sigma_k\setminus0)$. Moreover the principal symbol is non-vanishing on $N^\ast\mathcal{L}_{jk}\setminus0$. 
\end{itemize}
\end{theorem}
\begin{proof}[{\bf Proof}]
Recall Lemmas~\ref{theorem:intersection} and \ref{theorem:chid}. 
To prove (I), we apply Lemma~\ref{theorem:gu} to $\mathcal{R}_d\chi_{D_j}\cdot\mathcal{R}_d\chi_{D_k}$ with 
$$
N=N(d,n),
\quad
m_1=m_2=-\frac{d+2}{2}, 
\quad
k_1=k_2=1, 
\quad
l_1=l_2=1. 
$$
Then we deduce that 
\begin{align*}
& \mathcal{R}_d\chi_{D_j}\cdot\mathcal{R}_d\chi_{D_k}
\\
& \qquad \in 
  I^{-(d+2)/2+1/2-N(d,n)/4, -(d+2)/2+1/2}\bigl(N^\ast{S_{jk}},N^\ast{S_j}\bigr)
\\
& \qquad +
  I^{-(d+2)/2+1/2-N(d,n)/4, -(d+2)/2+1/2}\bigl(N^\ast{S_{jk}},N^\ast{S_k}\bigr) 
\\
& \qquad \in 
  I^{-(d+1)/2-N(d,n)/4, -(d+1)/2}\bigl(N^\ast{S_{jk}},N^\ast{S_j}\bigr)
\\
& \qquad +
  I^{-(d+1)/2-N(d,n)/4, -(d+1)/2}\bigl(N^\ast{S_{jk}},N^\ast{S_k}\bigr) 
\\
& \qquad =
  \mathcal{A}_{jk}. 
\end{align*}
To prove that the principal symbol of $\mathcal{R}_d\chi_{D_j}\cdot\mathcal{R}_d\chi_{D_k}$ does not vanish everywhere on $N^\ast{S_{jk}}$, we need to go back to the proof of Lemma~\ref{theorem:gu}. Here we describe its outline. Recall that $S_j$  and $S_k$ intersect transversely. Fix arbitrary $p_0 \in S_{jk}$. We can choose local coordinates $(x,y,z)\in\mathbb{R}\times\mathbb{R}\times\mathbb{R}^{N(d,n)-2}$ near $p_0$ with $\bigl(x(p_0),y(p_0),z(p_0)\bigr)=0$ such that $S_j=\{x=0\}$ and $S_k=\{y=0\}$ near $p_0$. Moreover there exist symbols $a(y,z,\xi)$ and $b(x,z,\eta)$ such that 
the principal part of $a(y,z,\xi)$ is non-vanishing on $N^\ast_{p_0}{S_j}=\{(0,y,z;\xi,0,0,)\}$, 
the principal part of $b(x,z,\eta)$ is non-vanishing on $N^\ast_{p_0}{S_k}=\{(x,0,z,\xi,0,0)\}$, 
and  
\begin{align*}
  \mathcal{R}_d\chi_{D_j}(x,y,z)
& =
  \int_{\mathbb{R}}
  e^{ix\xi}
  a(y,z,\xi)
  d\xi, 
\\
  \mathcal{R}_d\chi_{D_k}(x,y,z)
& =
  \int_{\mathbb{R}}
  e^{iy\eta}
  b(x,z,\eta)
  d\eta 
\end{align*}
near $p_0$. Hence 
$$
\mathcal{R}_d\chi_{D_j}(x,y,z)
\cdot
\mathcal{R}_d\chi_{D_k}(x,y,z)
=
\iint_{\mathbb{R}\times\mathbb{R}}
e^{i(x\xi+y\eta)}
a(y,z,\xi)b(x,z,\eta)
d\xi
d\eta
$$
near $p_0$. This shows that the symbol of $\mathcal{R}_d\chi_{D_j}(x,y,z)\cdot\mathcal{R}_d\chi_{D_k}(x,y,z)$ is $a(y,z,\xi)b(x,z,\eta)$ whose principal part is non-vanishing on $N^\ast_{p_0}{S_{jk}}=\{(0,0,z;\xi,\eta,0)\}$ near $p_0$. 
\par
Next we consider (II). $\mathcal{R}_d^\ast(-\Delta_{x^{\prime\prime}})^{d/2}$ is an elliptic Fourier integral operator of order 
$$
\frac{d}{2}+\frac{N(d,n)}{4}-\frac{n}{4}, 
$$
and its canonical relation is the same as that of $\mathcal{R}_d^\ast$, that is, 
$(\Lambda_\phi^\prime)^\ast$ since $(-\Delta_{x^{\prime\prime}})^{d/2}$ is an elliptic pseudodifferential operator. 
Moreover Lemma~\ref{theorem:canonical2} shows that 
the excess of 
$(\Lambda_\phi^\prime){\circ}N^\ast{S}_j$ and 
$(\Lambda_\phi^\prime){\circ}N^\ast{S_{jk}^{(1)}}$ is $d(n-1-d)$, 
and the excess of $(\Lambda_\phi^\prime){\circ}N^\ast{S_{jk}^{(2)}}$ is $d(n-2-d)$. 
There is a difference $-d/2$ of the excesses. 
Here we remark about the action of Fourier integral operators 
on paired Lagrangian distributions. 
Suppose that 
$A$ is a Fourier integral operator of order $m$ with a canonical relation $C$, 
$u \in I^{\mu,\nu}(\Lambda_0,\Lambda_1)$, 
and $C{\circ}\Lambda_j$ is clean with excess $e_j$ for $j=0,1$. 
Then Lemma~\ref{theorem:hoermander2523} implies that  
$$
Au \in I^{m+\mu+e_1/2,\nu+(e_0-e_1)/2}(C{\circ}\Lambda_0,C{\circ}\Lambda_1). 
$$
Hence we deduce that
\begin{align*}
& \mathcal{R}_d^\ast(-\Delta_{x^{\prime\prime}})^{d/2}\bigl\{\mathcal{R}_d\chi_{D_j}(x,y,z)\cdot\mathcal{R}_d\chi_{D_k}\bigr\}
\\
& \qquad \in
  I^{-(d+1)/2-n/4+d(n-d)/2, -(d+1)/2}\bigl((\Lambda_\phi^\prime)^\ast{\circ}N^\ast{S_{jk}^{(1)}},(\Lambda_\phi^\prime)^\ast{\circ}N^\ast{S_j}\bigr)
\\
& \qquad +
  I^{-(d+1)/2-n/4+d(n-d)/2, -(d+1)/2}\bigl((\Lambda_\phi^\prime)^\ast{\circ}N^\ast{S_{jk}^{(1)}},(\Lambda_\phi^\prime)^\ast{\circ}N^\ast{S_k}\bigr) 
\\
& \qquad +
  I^{-(d+1)/2-n/4+d(n-d)/2, -d-1/2}\bigl((\Lambda_\phi^\prime)^\ast{\circ}N^\ast{S_{jk}^{(2)}},(\Lambda_\phi^\prime)^\ast{\circ}N^\ast{S_j}\bigr)
\\
& \qquad +
  I^{-(d+1)/2-n/4+d(n-d)/2, -d-1/2}\bigl((\Lambda_\phi^\prime)^\ast{\circ}N^\ast{S_{jk}^{(2)}},(\Lambda_\phi^\prime)^\ast{\circ}N^\ast{S_k}\bigr) 
\\
& \qquad =
  I^{-(d+1)/2-n/4+d(n-d)/2, -(d+1)/2}\bigl(N^\ast\mathcal{L}_{jk}\setminus0,N^\ast\Sigma_j\setminus0\bigr)
\\
& \qquad +
  I^{-(d+1)/2-n/4+d(n-d)/2, -(d+1)/2}\bigl(N^\ast\mathcal{L}_{jk}\setminus0,N^\ast\Sigma_k\setminus0\bigr) 
\\
& \qquad +
  I^{-(d+1)/2-n/4+d(n-d)/2, -d-1/2}\bigl((N^\ast\Omega_{jk}\setminus0){\cup}(N^\ast\Omega_{kj}\setminus0),N^\ast\Sigma_j\setminus0\bigr)
\\
& \qquad +
  I^{-(d+1)/2-n/4+d(n-d)/2, -d-1/2}\bigl((N^\ast\Omega_{jk}\setminus0){\cup}(N^\ast\Omega_{kj}\setminus0),N^\ast\Sigma_k\setminus0\bigr) 
\\
& \qquad \subset
  I^{-(d+1)/2-n/4+d(n-d)/2, -(d+1)/2}\bigl(N^\ast\mathcal{L}_{jk}\setminus0,N^\ast\Sigma_j\setminus0\bigr)
\\
& \qquad +
  I^{-(d+1)/2-n/4+d(n-d)/2, -(d+1)/2}\bigl(N^\ast\mathcal{L}_{jk}\setminus0,N^\ast\Sigma_k\setminus0\bigr) 
\\
& \qquad +
  I^{-(d+1)/2-n/4+d(n-d)/2}\bigl(N^\ast\Sigma_j\setminus0\bigr)
\\
& \qquad +
  I^{-(d+1)/2-n/4+d(n-d)/2}\bigl(N^\ast\Sigma_k\setminus0\bigr) 
\\
& \qquad \subset
  I^{-(d+1)/2-n/4+d(n-d)/2, -(d+1)/2}\bigl(N^\ast\mathcal{L}_{jk}\setminus0,N^\ast\Sigma_j\setminus0\bigr)
\\
& \qquad +
  I^{-(d+1)/2-n/4+d(n-d)/2, -(d+1)/2}\bigl(N^\ast\mathcal{L}_{jk}\setminus0,N^\ast\Sigma_k\setminus0\bigr) 
\\
& \qquad  =
  \mathcal{X}_{jk} 
\end{align*}
The principal part of the symbol of $\mathcal{R}_d^\ast(-\Delta_{x^{\prime\prime}})^{d/2}\bigl\{\mathcal{R}_d\chi_{D_j}(x,y,z)\cdot\mathcal{R}_d\chi_{D_k}(x,y,z)\bigr\}$ is non-vanishing on $N^\ast\mathcal{L}_{jk}\setminus0$ since the principal symbol of $\mathcal{R}_d\chi_{D_j}(x,y,z)\cdot\mathcal{R}_d\chi_{D_k}(x,y,z)$ is non-vanishing on $N^\ast{S_{jk}}$ and $\mathcal{R}_d^\ast(-\Delta_{x^{\prime\prime}})^{d/2}$ is an elliptic Fourier integral operator. This completes the proof.
\end{proof}
%
%
\section{Generalized beam hardening effects}
\label{section:artifacts}
Finally in the present section we prove our main theorem on what the metal streaking artifacts are. Recall that the CT image of the metal region is given by 
$$
f_\text{MA}
=
\sum_{l=1}^\infty
A_l
(\alpha\varepsilon)^{2l}
\mathcal{R}_d^\ast(-\Delta_{x^{\prime\prime}})^{d/2}
\bigl[(\mathcal{R}_d\chi_D)^{2l}\bigr],
$$
where $\{A_l\}_{l=1}^\infty$ is a sequence of real numbers, 
and $\alpha$ and $\varepsilon$ are small positive constants. 
Our main results are the following. 
 \begin{theorem}
\label{theorem:main} 
Suppose {\rm (A)} on the metal region $D$. Then away from $N^\ast\Sigma$, the metal streaking artifacts $f_\text{MA}$ belongs to $I^{-(d+1+n/4)+d(n-d)/2}(N^\ast\mathcal{L}\setminus0)$, where 
$$
N^\ast\mathcal{L}
:=
\bigcup_{1 \leqq j <\leqq J}
N^\ast\mathcal{L}_{jk}.
$$
\end{theorem}
Roughly speaking, Theorem~\ref{theorem:main} asserts that $f_\text{MA}$ is a conormal distribution whose singular support consists cone surfaces and cylinder surfaces contacting $\Sigma_1,\dotsc,\Sigma_J$. 
\par
Basically our strategy of the proof of Theorem~\ref{theorem:main} is the same as that of \cite[Theorem~4.7]{PalaciosUhlmannWang} due to Palacios, Uhlmann and Wang. Firstly we state our strategy of the  proof of Theorem~\ref{theorem:main} shortly. Set 
$$
\mathcal{A}
:=
\sum_{1 \leqq j < k \leqq J}
\mathcal{A}_{jk}, 
\quad
\mathcal{X}
:=
\sum_{1 \leqq j < k \leqq J}
\mathcal{X}_{jk}, 
$$
In terms of this notation Theorem~\ref{theorem:main} asserts 
that $f_\text{MA} \in \mathcal{X}$. 
Lemmas~\ref{theorem:canonical2}, \ref{theorem:lemma32} and \ref{theorem:quadratic} 
essentially show that 
$(\mathcal{R}_d\chi_D)^2 \in \mathcal{A}$, 
$\mathcal{R}_d^\ast(\mathcal{A}) \subset \mathcal{X}$, 
and that away from $N^\ast\Sigma$, 
$u \in I^{-(d+1+n/4)+d(n-d)/2}(N^\ast\mathcal{L}\setminus0)$ 
for $u\in\mathcal{A}$. 
If $\mathcal{A}$ is an algebra, then Theorem~\ref{theorem:main} holds. 
So we shall prove that $\mathcal{A}$ is an algebra. 
\begin{lemma}
\label{theorem:final}
Suppose that $j{\ne}k$. 
Then $\mathcal{A}_{jk}$ is an algebra, and so is $\mathcal{A}$. 
More concretely we have the following. 
\begin{itemize}
\item[{\rm (I)}] 
If 
$$
u,v 
\in 
I^{-(d+1)/2-N(d,n)/4, -(d+1)/2}\bigl(N^\ast{S_{jk}},N^\ast{S_j}\bigr),  
$$
then 
$$
uv \in I^{-(d+1)/2-N(d,n)/4, -(d+1)/2}\bigl(N^\ast{S_{jk}},N^\ast{S_j}\bigr).  
$$
\item[{\rm (II)}] 
If
$$
u 
\in 
I^{-(d+1)/2-N(d,n)/4, -(d+1)/2}\bigl(N^\ast{S_{jk}},N^\ast{S_j}\bigr), 
$$
$$
v
\in 
I^{-(d+1)/2-N(d,n)/4, -(d+1)/2}\bigl(N^\ast{S_{jk}},N^\ast{S_k}\bigr),  
$$
then $uv \in \mathcal{A}_{jk}$. 
\end{itemize}
\end{lemma}
In what follows, let $\psi(t)$ be the same as in the proof of Lemma~\ref{theorem:lemma32}. 
Firstly we prove (I) of Lemma~\ref{theorem:final}. 
\begin{proof}[{\bf Proof of Lemma~\ref{theorem:final} (I)}]
Suppose that 
$$
u,v 
\in 
I^{-(d+1)/2-N(d,n)/4, -(d+1)/2}\bigl(N^\ast{S_{jk}},N^\ast{S_j}\bigr). 
$$
Fix arbitrary $p_0 \in S_{jk}$, and choose appropriate local coordinates 
$(x,y,z)\in\mathbb{R}\times\mathbb{R}\times\mathbb{R}^{N(d,n)-2}$ such that 
$\bigl(x(p_0),y(p_0),z(p_0)\bigr)=0$, 
and 
\begin{align*}
  N^\ast{S}_j
& =
  \{(0,y,z;\xi,0,0) : \xi,y\in\mathbb{R}, z\in\mathbb{R}^{N(d,n)-2}\}, 
\\
  N^\ast{S}_k
& =
  \{(x,0,z;0,\eta,0) : x,\eta\in\mathbb{R}, z\in\mathbb{R}^{N(d,n)-2}\}, 
\\
  N^\ast{S}_{jk}
& =
  \{(0,0,z;\xi,\eta,0) : \xi,\eta\in\mathbb{R}, z\in\mathbb{R}^{N(d,n)-2}\}
\end{align*}
near $(x,y,z)=(0,0,0)$. 
Then there exist symbols 
$$
a(z,\xi,\eta), 
b(z,\xi,\eta) 
\in 
S^{-(d+2)/2,-(d+2)/2}\bigl(\mathbb{R}^{N(d,n)-2}\times(\mathbb{R}\setminus\{0\})\times\mathbb{R}\bigr)
$$
such that $a(z,\xi,\eta)$ and $b(z,\xi,\eta)$ are compactly supported in $z$ near $z=0$, 
and $u$ and $v$ are represented by oscillatory integrals
\begin{align*}
  u(x,y,z)
& =
  \iint_{\mathbb{R}\times\mathbb{R}}
  e^{i(x\xi+y\eta)}
  a(z,\xi,\eta)
  d\xi
  d\eta,
\\
  v(x,y,z)
& =
  \iint_{\mathbb{R}\times\mathbb{R}}
  e^{i(x\xi+y\eta)}
  b(z,\xi,\eta)
  d\xi
  d\eta
\end{align*}
near $(x,y,z)=(0,0,0)$. Then we have 
\begin{align*}
  (uv)(x,y,z)
& =
  \iint_{\mathbb{R}\times\mathbb{R}} 
  e^{i(x\xi+y\eta)}
  c_2(z,\xi,\eta)
  d\xi
  d\eta,
\\
  c_2(z,\xi,\eta)
& =
  \iint_{\mathbb{R}\times\mathbb{R}}
  a(z,\xi-\zeta,\eta-\tau)
  b(z,\zeta,\tau)
  d\zeta
  d\tau
\end{align*}
near $(x,y,z)=(0,0,0)$. 
Note that $c_2(z,\xi,\eta)$ is compactly supported in $z$ near $z=0$. 
It suffices to show that 
$$
c_2(z,\xi,\eta)
\in 
S^{-(d+2)/2,-(d+2)/2}\bigl(\mathbb{R}^{N(d,n)-2}\times(\mathbb{R}\setminus\{0\})\times\mathbb{R}\bigr). 
$$
For this purpose we introduce two cut-off functions defined by $\psi(t)$: 
\begin{align*}
  \Psi_2(\xi,\eta,\zeta,\tau)
& =
  \psi\left(\frac{\langle\xi-\zeta;\eta-\tau\rangle}{\langle\xi;\eta\rangle}\right), 
\\
  \Psi_3(\eta,\tau)
& =
  \psi\left(\frac{\langle\eta-\tau\rangle}{\langle\eta\rangle}\right).  
\end{align*}
In the same way as $\Psi_1$ in the Proof of Lemma~\ref{theorem:lemma32}, 
we can deduce the properties of $\Psi_2$ and $\Psi_3$ as follows. 
\begin{itemize}
\item[(i)] 
If $\Psi_2(\xi,\eta,\zeta,\tau)>0$, then $C_1^{-1}\langle\xi;\eta\rangle\leqq\langle\zeta;\tau\rangle\leqq{C_1}\langle\xi;\eta\rangle$ and 
$$
\partial_\zeta^\alpha\partial_\tau^\beta\Psi_2(\xi,\eta,\zeta,\tau)=\mathcal{O}(\langle\xi;\eta\rangle^{-(\alpha+\beta)})
$$ 
for all non-negative integers $\alpha$ and $\beta$. 
\item[(ii)] 
$1-\Psi_2(\xi,\eta,\zeta,\tau)>0$ is equivalent to 
$\langle\xi-\zeta;\eta-\tau\rangle>\langle\xi;\eta\rangle/2$. 
\item[(iii)] 
If $\Psi_3(\eta,\tau)>0$, then $C_1^{-1}\langle\eta\rangle\leqq\langle\tau\rangle\leqq{C_1}\langle\eta\rangle$ and $\partial_\tau^\beta\Psi_3(\eta,\tau)=\mathcal{O}(\langle\eta\rangle^{-\beta})$ for any non-negative integer $\beta$. 
\item[(iv)] 
$1-\Psi_3(\eta,\tau)>0$ is equivalent to 
$\langle\eta-\tau\rangle>\langle\eta\rangle/2$. 
\end{itemize}
We make use of 
$$
1=\Psi_2\Psi_3+\Psi_2(1-\Psi_3)+(1-\Psi_2)\Psi_3+(1-\Psi_2)(1-\Psi_3),
$$
$$
\frac{\partial}{\partial\xi}(\xi-\zeta)
=
-\frac{\partial}{\partial\zeta}(\xi-\zeta),
\quad
\frac{\partial}{\partial\eta}(\eta-\tau)
=
-\frac{\partial}{\partial\tau}(\eta-\tau), 
$$
and integration by parts. We split 
$\partial_z^\gamma\partial_\xi^\alpha\partial_\eta^\beta c_2(z,\xi,\eta)$ 
into four parts:
\begin{align}
  \partial_z^\gamma\partial_\xi^\alpha\partial_\eta^\beta c_2(z,\xi,\eta)
& =
  \sum_{\gamma_1+\gamma_2=\gamma}
  \frac{\gamma!}{\gamma_1!\gamma_2!}
  \iint_{\mathbb{R}\times\mathbb{R}}
  \partial_z^{\gamma_1}\partial_\xi^\alpha\partial_\eta^\beta a(z,\xi-\zeta,\eta-\tau)
\nonumber
\\
& \qquad\qquad\qquad\qquad \times
  \partial_z^{\gamma_2} b(z,\zeta,\eta) 
  d\zeta
  d\tau
\nonumber
\\
& = F_2+G_2+H_2+I_2,
\label{equation:goal1}
\\
  F_2
& =
  \sum_{\gamma_1+\gamma_2=\gamma}
  \frac{\gamma!}{\gamma_1!\gamma_2!}
  \iint_{\mathbb{R}\times\mathbb{R}}
  \partial_z^{\gamma_1}(-\partial_\zeta)^\alpha(-\partial_\tau)^\beta a(z,\xi-\zeta,\eta-\tau)
\nonumber
\\
& \qquad\qquad\qquad\qquad\times
  \Psi_2\Psi_3\partial_z^{\gamma_2} b(z,\zeta,\eta) 
  d\zeta
  d\tau,
\nonumber
\\
  G_2
& =
  \sum_{\gamma_1+\gamma_2=\gamma}
  \frac{\gamma!}{\gamma_1!\gamma_2!}
  \iint_{\mathbb{R}\times\mathbb{R}}
  \partial_z^{\gamma_1}(-\partial_\zeta)^\alpha\partial_\eta^\beta a(z,\xi-\zeta,\eta-\tau)
\nonumber
\\
& \qquad\qquad\qquad\qquad\times
  \Psi_2(1-\Psi_3)\partial_z^{\gamma_2} b(z,\zeta,\eta) 
  d\zeta
  d\tau,
\nonumber
\\
  H_2
& =
  \sum_{\gamma_1+\gamma_2=\gamma}
  \frac{\gamma!}{\gamma_1!\gamma_2!}
  \iint_{\mathbb{R}\times\mathbb{R}}
  \partial_z^{\gamma_1}\partial_\xi^\alpha(-\partial_\tau)^\beta a(z,\xi-\zeta,\eta-\tau)
\nonumber
\\
& \qquad\qquad\qquad\qquad\times
  (1-\Psi_2)\Psi_3\partial_z^{\gamma_2} b(z,\zeta,\eta) 
  d\zeta
  d\tau,
\nonumber
\\
  I_2
& =
  \sum_{\gamma_1+\gamma_2=\gamma}
  \frac{\gamma!}{\gamma_1!\gamma_2!}
  \iint_{\mathbb{R}\times\mathbb{R}}
  \partial_z^{\gamma_1}\partial_\xi^\alpha\partial_\eta^\beta a(z,\xi-\zeta,\eta-\tau)
\nonumber
\\
& \qquad\qquad\qquad\qquad\times
  (1-\Psi_2)(1-\Psi_3)\partial_z^{\gamma_2} b(z,\zeta,\eta) 
  d\zeta
  d\tau.
\nonumber
\end{align}
The region of the integration $F_2$ satisfies 
$$
C_1^{-1}\langle\xi;\eta\rangle\leqq\langle\zeta;\tau\rangle\leqq{C_1}\langle\xi;\eta\rangle,
\quad
C_1^{-1}\langle\eta\rangle\leqq\langle\tau\rangle\leqq{C_1}\langle\eta\rangle. 
$$
We deduce that 
\begin{align}
  F_2
& =
  \sum_{\alpha_2+\alpha_1=\alpha}
  \frac{\alpha!}{\alpha_2!\alpha_1!}
  \sum_{\beta_2+\beta_3+\beta_1=\beta}
  \frac{\beta!}{\beta_2!\beta_3!\beta_1!}
  \sum_{\gamma_1+\gamma_2=\gamma}
  \frac{\gamma!}{\gamma_1!\gamma_2!}
\nonumber
\\
& \times
  \iint_{\substack{C_1^{-1}\langle\xi;\eta\rangle\leqq\langle\zeta;\tau\rangle\leqq{C_1}\langle\xi;\eta\rangle \\ C_1^{-1}\langle\eta\rangle\leqq\langle\tau\rangle\leqq{C_1}\langle\eta\rangle}}
  \partial_z^{\gamma_1}a(z,\xi-\zeta,\eta-\tau)
\nonumber
\\
& \times
  \partial_\zeta^{\alpha_2}\partial_\tau^{\beta_2}\Psi_2
  \cdot
  \partial_\tau^{\beta_3}\Psi_3
  \cdot
  \partial_z^{\gamma_2}\partial_\zeta^{\alpha_1}\partial_\tau^{\beta_1} b(z,\zeta,\eta) 
  d\zeta
  d\tau
\nonumber
\\
& =
  \sum_{\alpha_2+\alpha_1=\alpha}
  \sum_{\beta_2+\beta_3+\beta_1=\beta}
  \iint_{\substack{C_1^{-1}\langle\xi;\eta\rangle\leqq\langle\zeta;\tau\rangle\leqq{C_1}\langle\xi;\eta\rangle \\ C_1^{-1}\langle\eta\rangle\leqq\langle\tau\rangle\leqq{C_1}\langle\eta\rangle}}
\nonumber
\\
& \times
  \mathcal{O}(\langle\xi-\zeta;\eta-\tau\rangle^{-(d+2)/2}\langle\eta-\tau\rangle^{-(d+2)/2} \langle\xi;\eta\rangle^{-\alpha_2-\beta_2} 
\nonumber
\\
& \times
  \langle\eta\rangle^{-\beta_3} \langle\zeta;\tau\rangle^{-(d+2)/2-\alpha_1}\langle\tau\rangle^{-(d+2)/2-\beta_1})
  d\zeta
  d\tau
\nonumber
\\
&  =
  \iint_{\mathbb{R}\times\mathbb{R}}
  \mathcal{O}(\langle\xi-\zeta\rangle^{-(d+2)/2}\langle\eta-\tau\rangle^{-(d+2)/2} \langle\xi;\eta\rangle^{-(d+2)/2-\alpha}\langle\eta\rangle^{-(d+2)/2-\beta})
  d\zeta
  d\tau
\nonumber
\\
& =
  \mathcal{O}(\langle\xi;\eta\rangle^{-(d+2)/2-\alpha}\langle\eta\rangle^{-(d+2)/2-\beta}). 
\label{equation:part11}
\end{align}
since $(d+2)/2=1+d/2>1$ guarantees that 
$\langle\xi-\zeta\rangle^{-(d+2)/2}$ is integrable in $\zeta$ on $\mathbb{R}$, and 
$\langle\eta-\tau\rangle^{-(d+2)/2}$ is integrable in $\tau$ on $\mathbb{R}$. 
\par
The region of the integration $G_2$ satisfies
$$
C_1^{-1}\langle\xi;\eta\rangle\leqq\langle\zeta;\tau\rangle\leqq{C_1}\langle\xi;\eta\rangle,
\quad 
\langle\eta\rangle/2\leqq\langle\eta-\tau\rangle.
$$
We deduce that 
\begin{align}
  G_2
& =
  \sum_{\alpha_2+\alpha_1=\alpha}
  \frac{\alpha!}{\alpha_2!\alpha_1!}
  \sum_{\gamma_1+\gamma_2=\gamma}
  \frac{\gamma!}{\gamma_1!\gamma_2!}
\nonumber
\\
& \times
  \iint_{\substack{C_1^{-1}\langle\xi;\eta\rangle\leqq\langle\zeta;\tau\rangle\leqq{C_1}\langle\xi;\eta\rangle \\ \langle\eta\rangle/2\leqq\langle\eta-\tau\rangle}}
  \partial_z^{\gamma_1}\partial_\eta^\beta a(z,\xi-\zeta,\eta-\tau)
\nonumber
\\
& \times
  \partial_\zeta^{\alpha_2}\Psi_2
  \cdot
  (1-\Psi_3)
  \cdot
  \partial_\zeta^{\alpha_1}\partial_z^{\gamma_2} b(z,\zeta,\eta) 
  d\zeta
  d\tau
\nonumber
\\
& =
  \sum_{\alpha_2+\alpha_1=\alpha}
  \iint_{\substack{C_1^{-1}\langle\xi;\eta\rangle\leqq\langle\zeta;\tau\rangle\leqq{C_1}\langle\xi;\eta\rangle \\ \langle\eta\rangle/2\leqq\langle\eta-\tau\rangle}}
\nonumber
\\
& \times
  \mathcal{O}(\langle\xi-\zeta;\eta-\tau\rangle^{-(d+2)/2}\langle\eta-\tau\rangle^{-(d+2)/2-\beta} 
\nonumber
\\
& \times
  \langle\xi;\eta\rangle^{-\alpha_2} \langle\zeta;\tau\rangle^{-(d+2)/2-\alpha_1}\langle\tau\rangle^{-(d+2)/2})
  d\zeta
  d\tau
\nonumber
\\
& =
  \iint_{\mathbb{R}\times\mathbb{R}}
  \mathcal{O}(\langle\xi-\zeta\rangle^{-(d+2)/2}\langle\eta\rangle^{-(d+2)/2-\beta} \langle\xi;\eta\rangle^{-(d+2)/2-\alpha}\langle\tau\rangle^{-(d+2)/2})
  d\zeta
  d\tau
\nonumber
\\
& =
  \mathcal{O}(\langle\xi;\eta\rangle^{-(d+2)/2-\alpha}\langle\eta\rangle^{-(d+2)/2-\beta})
\label{equation:part12}
\end{align}
since $(d+2)/2=1+d/2>1$ guarantees that 
$\langle\xi-\zeta\rangle^{-(d+2)/2}$ is integrable in $\zeta$ on $\mathbb{R}$, and 
$\langle\tau\rangle^{-(d+2)/2}$ is integrable in $\tau$ on $\mathbb{R}$. 
\par
The region of the integration $H_2$ satisfies 
$$
\langle\xi;\eta\rangle/2\leqq\langle\xi-\zeta;\eta-\tau\rangle, 
\quad
C_1^{-1}\langle\eta\rangle\leqq\langle\tau\rangle\leqq{C_1}\langle\eta\rangle.
$$
We deduce that
\begin{align}
  H_2
& =
  \sum_{\beta_2+\beta_3+\beta_1=\beta}
  \frac{\beta!}{\beta_2!\beta_3!\beta_1!}
  \sum_{\gamma_1+\gamma_2=\gamma}
  \frac{\gamma!}{\gamma_1!\gamma_2!}
\nonumber
\\
& \times
  \iint_{\substack{\langle\xi;\eta\rangle/2\leqq\langle\xi-\zeta;\eta-\tau\rangle \\ C_1^{-1}\langle\eta\rangle\leqq\langle\tau\rangle\leqq{C_1}\langle\eta\rangle}}
  \partial_z^{\gamma_1}\partial_\xi^\alpha a(z,\xi-\zeta,\eta-\tau)
\nonumber
\\
& \times
  \partial_\tau^{\beta_2}(1-\Psi_2)
  \cdot
  \partial_\tau^{\beta_3}\Psi_3
  \cdot
  \partial_z^{\gamma_2}\partial_\tau^{\beta_1} b(z,\zeta,\eta) 
  d\zeta
  d\tau
\nonumber
\\
& =
  \sum_{\beta_2+\beta_3+\beta_1=\beta}
  \iint_{\substack{\langle\xi;\eta\rangle/2\leqq\langle\xi-\zeta;\eta-\tau\rangle \\ C_1^{-1}\langle\eta\rangle\leqq\langle\tau\rangle\leqq{C_1}\langle\eta\rangle}}
\nonumber
\\
& \times
  \mathcal{O}(\langle\xi-\zeta;\eta-\tau\rangle^{-(d+2)/2-\alpha}\langle\eta-\tau\rangle^{-(d+2)/2} 
\nonumber
\\
& \times
  \langle\xi;\eta\rangle^{-\beta_2} \langle\eta\rangle^{-\beta_3} \langle\zeta;\tau\rangle^{-(d+2)/2}\langle\tau\rangle^{-(d+2)/2-\beta_1})
  d\zeta
  d\tau
\nonumber
\\
& =
  \iint_{\mathbb{R}\times\mathbb{R}}
  \mathcal{O}(\langle\xi;\eta\rangle^{-(d+2)/2-\alpha}\langle\eta-\tau\rangle^{-(d+2)/2} \langle\zeta\rangle^{-(d+2)/2}\langle\eta\rangle^{-(d+2)/2-\beta})
  d\zeta
  d\tau
\nonumber
\\
& =
  \mathcal{O}(\langle\xi;\eta\rangle^{-(d+2)/2-\alpha}\langle\eta\rangle^{-(d+2)/2-\beta})
\label{equation:part13}
\end{align}
since $(d+2)/2=1+d/2>1$ guarantees that 
$\langle\zeta\rangle^{-(d+2)/2}$ is integrable in $\zeta$ on $\mathbb{R}$, and 
$\langle\eta-\tau\rangle^{-(d+2)/2}$ is integrable in $\tau$ on $\mathbb{R}$. 
\par
The region of the integration $I_2$ satisfies
$$
\langle\xi;\eta\rangle/2\leqq\langle\xi-\zeta;\eta-\tau\rangle, 
\quad
\langle\eta\rangle/2\leqq\langle\eta-\tau\rangle. 
$$
We deduce that 
\begin{align}
  I_2
& =
  \sum_{\gamma_1+\gamma_2=\gamma}
  \frac{\gamma!}{\gamma_1!\gamma_2!}
  \iint_{\substack{\langle\xi;\eta\rangle/2\leqq\langle\xi-\zeta;\eta-\tau\rangle \\ \langle\eta\rangle/2\leqq\langle\eta-\tau\rangle}}
  \partial_z^{\gamma_1}\partial_\xi^\alpha\partial_\eta^\beta a(z,\xi-\zeta,\eta-\tau)
\nonumber
\\
& \times
  (1-\Psi_2)\cdot(1-\Psi_3)\cdot\partial_z^{\gamma_2} b(z,\zeta,\eta) 
  d\zeta
  d\tau
\nonumber
\\
& =
  \iint_{\substack{\langle\xi;\eta\rangle/2\leqq\langle\xi-\zeta;\eta-\tau\rangle \\ \langle\eta\rangle/2\leqq\langle\eta-\tau\rangle}}
  \mathcal{O}(\langle\xi-\zeta;\eta-\tau\rangle^{-(d+2)/2-\alpha}\langle\eta-\tau\rangle^{-(d+2)/2-\beta} 
\nonumber
\\
& \times
  \langle\zeta;\tau\rangle^{-(d+2)/2}\langle\tau\rangle^{-(d+2)/2})
  d\zeta
  d\tau
\nonumber
\\
& =
  \iint_{\mathbb{R}\times\mathbb{R}}
  \mathcal{O}(\langle\xi;\eta\rangle^{-(d+2)/2-\alpha}\langle\eta\rangle^{-(d+2)/2-\beta} \langle\zeta\rangle^{-(d+2)/2}\langle\tau\rangle^{-(d+2)/2})
  d\zeta
  d\tau
\nonumber
\\
& =
  \mathcal{O}(\langle\xi;\eta\rangle^{-(d+2)/2-\alpha}\langle\eta\rangle^{-(d+2)/2-\beta})
\label{equation:part14}
\end{align}
since $(d+2)/2=1+d/2>1$ guarantees that 
$\langle\zeta\rangle^{-(d+2)/2}$ is integrable in $\zeta$ on $\mathbb{R}$, and 
$\langle\tau\rangle^{-(d+2)/2}$ is integrable in $\tau$ on $\mathbb{R}$. 
\par 
Substitute 
\eqref{equation:part11}, 
\eqref{equation:part12}, 
\eqref{equation:part13} 
and 
\eqref{equation:part14} 
into 
\eqref{equation:goal1}. 
We have 
$$
\partial_z^\gamma\partial_\xi^\alpha\partial_\eta^\beta c_2(z,\xi,\eta)
=
\mathcal{O}(\langle\xi;\eta\rangle^{-(d+2/2)-\alpha}\langle\eta\rangle^{-(d+2)/2-\beta}).
$$
This completes the proof.  
\end{proof}
Finally we prove (II) of Lemma~\ref{theorem:final}. 
\begin{proof}[{\bf Proof of Lemma~\ref{theorem:final} (II)}]
Suppose that 
$$
u 
\in 
I^{-(d+1)/2-N(d,n)/4, -(d+1)/2}\bigl(N^\ast{S_{jk}},N^\ast{S_j}\bigr), 
$$
$$
v
\in 
I^{-(d+1)/2-N(d,n)/4, -(d+1)/2}\bigl(N^\ast{S_{jk}},N^\ast{S_k}\bigr).  
$$
In the same way as the proof of (I), fix arbitrary $p_0 \in S_{jk}$, 
and choose appropriate local coordinates 
$(x,y,z)\in\mathbb{R}\times\mathbb{R}\times\mathbb{R}^{N(d,n)-2}$ such that 
$\bigl(x(p_0),y(p_0),z(p_0)\bigr)=0$, 
and 
\begin{align*}
  N^\ast{S}_j
& =
  \{(0,y,z;\xi,0,0) : \xi,y\in\mathbb{R}, z\in\mathbb{R}^{N(d,n)-2}\}, 
\\
  N^\ast{S}_k
& =
  \{(x,0,z;0,\eta,0) : x,\eta\in\mathbb{R}, z\in\mathbb{R}^{N(d,n)-2}\}, 
\\
  N^\ast{S}_{jk}
& =
  \{(0,0,z;\xi,\eta,0) : \xi,\eta\in\mathbb{R}, z\in\mathbb{R}^{N(d,n)-2}\}
\end{align*}
near $(x,y,z)=(0,0,0)$. 
Then there exist symbols 
$$
a(z,\xi,\eta), 
b(z,\eta,\xi) 
\in 
S^{-(d+2)/2,-(d+2)/2}\bigl(\mathbb{R}^{N(d,n)-2}\times(\mathbb{R}\setminus\{0\})\times\mathbb{R}\bigr)
$$
such that $a(z,\xi,\eta)$ and $b(z,\eta,\xi)$ are compactly supported in $z$ near $z=0$, 
and $u$ and $v$ are represented by oscillatory integrals
\begin{align*}
  u(x,y,z)
& =
  \iint_{\mathbb{R}\times\mathbb{R}}
  e^{i(x\xi+y\eta)}
  a(z,\xi,\eta)
  d\xi
  d\eta,
\\
  v(x,y,z)
& =
  \iint_{\mathbb{R}\times\mathbb{R}}
  e^{i(x\xi+y\eta)}
  b(z,\eta,\xi)
  d\xi
  d\eta
\end{align*}
near $(x,y,z)=(0,0,0)$. 
We need to take care of the order of the variables of $b(z,\eta,\xi)$. 
Then we have 
\begin{align*}
  (uv)(x,y,z)
& =
  \iint_{\mathbb{R}\times\mathbb{R}} 
  e^{i(x\xi+y\eta)}
  c_3(z,\xi,\eta)
  d\xi
  d\eta,
\\
  c_3(z,\xi,\eta)
& =
  \iint_{\mathbb{R}\times\mathbb{R}}
  a(z,\xi-\zeta,\eta-\tau)
  b(z,\tau,\zeta)
  d\zeta
  d\tau
\end{align*}
near $(x,y,z)=(0,0,0)$. 
Note that $c_3(z,\xi,\eta)$ is compactly supported in $z$ near $z=0$. 
It suffices to show that 
\begin{align*}
  c_3(z,\xi,\eta)
& \in 
  S^{-(d+2)/2,-(d+2)/2}\bigl(\mathbb{R}_z^{N(d,n)-2}\times(\mathbb{R_\xi}\setminus\{0\})\times\mathbb{R_\eta}\bigr)
\\
& +
  S^{-(d+2)/2,-(d+2)/2}\bigl(\mathbb{R}_z^{N(d,n)-2}\times(\mathbb{R}_\eta\setminus\{0\})\times\mathbb{R}_\xi\bigr).
\end{align*}
In the same way as the proof of (I), we can obtain
\begin{equation}
\partial_z^\gamma
\partial_\xi^\alpha
\partial_\eta^\beta
c_3(z,\xi,\eta)
=
\mathcal{O}(\langle\xi\rangle^{-(d+2)/2-\alpha}\langle\eta\rangle^{-(d+2)/2-\beta}). 
\label{equation:part20}
\end{equation}
To complete the proof we split $c_3(z,\xi,\eta)$ into two parts smoothly. 
For this purpose we introduce two cut-off functions defined by $\psi(t)$. 
Set $\Psi_4(\xi,\eta):=\psi(\langle\xi\rangle/\langle\eta\rangle)$. 
We obtain the following properties. 
\begin{itemize}
\item[(i)] 
If $\Psi_4(\xi,\eta)>0$, 
then $\langle\xi;\eta\rangle \leqq C_2\langle\eta\rangle$ with some $C_2>1$ 
since $\langle\xi\rangle \leqq 3\langle\eta\rangle/4$.  
\item[(ii)] 
If $1-\Psi_4(\xi,\eta)>0$, 
then $\langle\xi;\eta\rangle \leqq C_2\langle\xi\rangle$ with some $C_2>1$ 
since $\langle\xi\rangle \geqq \langle\eta\rangle/2$.
\item[(iii)] 
We have that 
$$
\partial_\xi^\alpha\partial_\eta^\beta\Psi_4(\xi,\eta), 
\partial_\xi^\alpha\partial_\eta^\beta\bigl(1-\Psi_4(\xi,\eta)\bigr)
=
\mathcal{O}(\langle\xi;\eta\rangle^{-\alpha-\beta})
$$ 
for all the non-negative integers $\alpha$ and $\beta$ since $0\leqq\Psi_4(\xi,\eta)\leqq1$ for all $(\xi,\eta)\in\mathbb{R}^2$. 
Moreover 
$\langle\xi;\eta\rangle \leqq C_2\langle\eta\rangle$ 
and  
$\langle\xi;\eta\rangle \leqq C_2\langle\xi\rangle$ 
hold in 
$\operatorname{supp}\psi^\prime(\langle\xi\rangle/\langle\eta\rangle)$.    
\end{itemize}
Set 
$c_{3,1}(z,\xi,\eta):=\bigl(1-\Psi_4(\xi,\eta)\bigr)c(z,\xi,\eta)$ 
and 
$c_{3,2}(z,\xi,\eta):=\Psi_4(\xi,\eta)c(z,\xi,\eta)$. 
Combining \eqref{equation:part20} and the properties (i), (ii) and (iii) above, 
we shall show 
\begin{align*}
  c_{3,1}(z,\xi,\eta)
& \in 
  S^{-(d+2)/2,-(d+2)/2}\bigl(\mathbb{R}_z^{N(d,n)-2}\times(\mathbb{R_\xi}\setminus\{0\})\times\mathbb{R_\eta}\bigr)
\\
  c_{3,2}(z,\xi,\eta)
& \in
  S^{-(d+2)/2,-(d+2)/2}\bigl(\mathbb{R}_z^{N(d,n)-2}\times(\mathbb{R}_\eta\setminus\{0\})\times\mathbb{R}_\xi\bigr).
\end{align*}
Since $\operatorname{supp}c_{3,1}(z,\cdot,\cdot)$ is contained in $\{\langle\xi;\eta\rangle \leqq C_2\langle\xi\rangle\}$, we deduce that 
\begin{align*}
  \partial_z\partial_\xi^\alpha\partial_\eta^\beta
  c_{3,1}(z,\xi,\eta)
& =
  \sum_{\alpha_1+\alpha_2=\alpha}
  \sum_{\beta_1+\beta_2=\beta}
  \frac{\alpha!}{\alpha_1!\alpha_2!}
  \frac{\beta!}{\beta_1!\beta_2!}
\\
& \times
  \partial_\xi^{\alpha_1}\partial_\eta^{\beta_1}
  \bigl(1-\Psi_4(\xi,\eta)\bigr)
  \cdot
  \partial_z\partial_\xi^{\alpha_2}\partial_\eta^{\beta_2}
  c_3(z,\xi,\eta)
\\
& =
  \sum_{\alpha_1+\alpha_2=\alpha}
  \sum_{\beta_1+\beta_2=\beta}
\\
& \times
  \mathcal{O}
  (\langle\xi;\eta\rangle^{-\alpha_1-\beta_1}\langle\xi\rangle^{-(d+2)/2-\alpha_2}\langle\eta\rangle^{-(d+2)/2-\beta_2})
\\
& =
  \mathcal{O}
  (\langle\xi;\eta\rangle^{-(d+2)/2-\alpha}\langle\eta\rangle^{-(d+2)/2-\beta}). 
\end{align*}
Since $\operatorname{supp}c_{3,2}(z,\cdot,\cdot)$ is contained in $\{\langle\xi;\eta\rangle \leqq C_2\langle\eta\rangle\}$, we deduce that 
\begin{align*}
  \partial_z\partial_\xi^\alpha\partial_\eta^\beta
  c_{3,2}(z,\xi,\eta)
& =
  \sum_{\alpha_1+\alpha_2=\alpha}
  \sum_{\beta_1+\beta_2=\beta}
  \frac{\alpha!}{\alpha_1!\alpha_2!}
  \frac{\beta!}{\beta_1!\beta_2!}
\\
& \times
  \partial_\xi^{\alpha_1}\partial_\eta^{\beta_1}\Psi_4(\xi,\eta)
  \cdot
  \partial_z\partial_\xi^{\alpha_2}\partial_\eta^{\beta_2}
  c_3(z,\xi,\eta)
\\
& =
  \sum_{\alpha_1+\alpha_2=\alpha}
  \sum_{\beta_1+\beta_2=\beta}
\\
& \times
  \mathcal{O}
  (\langle\xi;\eta\rangle^{-\alpha_1-\beta_1}\langle\xi\rangle^{-(d+2)/2-\alpha_2}\langle\eta\rangle^{-(d+2)/2-\beta_2})
\\
& =
  \mathcal{O}
  (\langle\xi;\eta\rangle^{-(d+2)/2-\beta}\langle\xi\rangle^{-(d+2)/2-\alpha}). 
\end{align*}
This completes the proof.   
\end{proof}
%
%
\begin{center}
{\sc Acknowledgments} 
\end{center}
\par
The author would like to thank the referees who carefully read the manuscript and provided many valuable comments and suggestions for improving this manuscript.

\end{document}